\documentclass[reqno,11pt]{amsart}
\usepackage{amsmath,amssymb,amsthm,graphicx,a4wide,enumerate,url}
\usepackage[small,bf]{caption} 
\usepackage{subcaption}
\usepackage{amsmath}
\usepackage{color}
\usepackage{lipsum}
\usepackage{amsfonts}
\usepackage{graphicx}
\usepackage{epstopdf}
\usepackage{algorithmic}
\usepackage[outercaption]{sidecap}
\usepackage{amsopn}
\usepackage{enumerate}
\usepackage{array}
\usepackage{makecell}
\usepackage{adjustbox}
\usepackage{calc}
\usepackage{accents}
\usepackage{hyperref}
\usepackage{cleveref}
\usepackage{romannum}
\usepackage{float}
\usepackage{multirow}
\usepackage{algorithm2e}
\usepackage{array, boldline, makecell, booktabs}

\usepackage{xcolor}
\hypersetup{
  colorlinks   = true, %Colours links instead of ugly boxes
  urlcolor     = blue, %Colour for external hyperlinks
  linkcolor    = blue, %Colour of internal links
  citecolor   = red %Colour of citations
}

\usepackage{bm}
\usepackage{tikz}
\usetikzlibrary{positioning}
\usepackage{tcolorbox}
\usepackage{appendix}

\AtBeginDocument{\pagenumbering{arabic}}

\newtheorem{thm}{Theorem}
\newtheorem{defn}{Definition}

\newtheorem{lemma}{Lemma}
\newtheorem{remark}{Remark}

\newcommand{\tIcim}{\bm{c}}
\newcommand{\tIAim}{A}
\newcommand{\tIbim}{\bm{b}}
\newcommand{\tIcex}{\tilde{\tIcim}}
\newcommand{\tIAex}{\tilde{A}}
\newcommand{\tIbex}{\tilde{\tIbim}}

\newcommand{\uh}{\hat{u}}
\newcommand{\vh}{\hat{v}}
\newcommand{\wh}{\hat{w}}
\newcommand{\qhat}{\hat{q}}
\newcommand{\gh}{\hat{g}}
\newcommand{\Yh}{\hat{Y}}

\DeclareMathOperator{\sech}{sech}
\newcommand{\order}{{\mathcal O}}
\newcommand{\ukdv}{\eta}
\newcommand{\qh}{\bm{q}}
\newcommand{\Lh}{\bm{L}}
\newcommand{\Nh}{\bm{N}}

\begin{document}
\parskip.9ex

%===========================================================================
%=================================================================== Titles.
\title[AP Exponential Integrators for the KdVH System]
{Asymptotic-Preserving Exponential Integrators applied to the Hyperbolic Korteweg-de Vries System}

\author[A. Biswas]{Abhijit Biswas}
\address[Abhijit Biswas]
{Department of Mathematics, Indian Institute of Technology Kanpur, Kanpur 208016, India} 
\email{abhijit@iitk.ac.in}
%\urladdr{https://math.temple.edu/\~{}tug14809}

\author[S. Boscarino]{Sebastiano Boscarino}
\address[Sebastiano Boscarino]
{Department of Mathematics and Computer Science, University of Catania, Viale Andrea Doria 6, 95125 Catania, Italy} 
\email{sebastiano.boscarino@unict.it}

\author[D. I. Ketcheson]{David I. Ketcheson}
\address[David I. Ketcheson]
{Computer, Electrical, and Mathematical Sciences \& Engineering Division \\ 
King Abdullah University of Science and Technology \\ Thuwal 23955 \\ Saudi Arabia} 
\email{david.ketcheson@kaust.edu.sa}
\urladdr{https://www.davidketcheson.info}

\author[G. Russo]{Giovanni Russo}
\address[Giovanni Russo]
{Department of Mathematics and Computer Science, University of Catania,
Viale Andrea Doria 6, 95125 Catania, Italy} 
\email{giovanni.russo@unict.it}


%===========================================================================
\begin{abstract}
We study the application of exponential time integration methods, coupled with Fourier pseudospectral
space discretization, to the hyperbolic Korteweg-de Vries (KdVH) system.  We investigate
the asymptotic preserving (AP) properties of such discretizations,
showing that, in general, Lawson methods are not asymptotic preserving, because the auxiliary (derivative-approximating) variables do not satisfy the limit equilibrium manifold, whereas exponential time differencing (ETD) methods are AP for all components, including both the solution variable and the auxiliary variables.
We also present an efficient numerical implementation based on an exact formula for the
matrix exponential of this system, and compare it with previously-proposed ImEx Runge-Kutta
time integration, showing that exponential methods can be competitive in this context.
\end{abstract}
%===========================================================================

\maketitle
%---------------------%
\section{Introduction}
%---------------------%
Recently, many authors have proposed systems of hyperbolic partial differential equations
(PDEs) that are designed to approximate higher-order PDEs (see e.g. 
\cite{toro2014advection,mazaheri2016first,chesnokov2019hyperbolic,escalante2019efficient,surov2020hyperbolization,gavrilyuk2022hyperbolic,besse2022perfectly,guermond2022hyperbolic,bleecke2025asymptotic,ketcheson2025approximation,biswas2025hyperbolic}).
These systems are based on
the use of a relaxation parameter, such that the solution of the hyperbolic system tends
to that of the original system in the relaxation limit.  One potential advantage of the
hyperbolic system is its use in numerical discretization, since it avoids the stiffness
that accompanies high-order spatial derivative terms.  However, as the relaxation parameter
tends to zero, the wave speed of the hyperbolic system tends to infinity, so that a new kind of
stiffness appears and must be dealt with if a high-accuracy approximation is required.

One example of such a hyperbolic relaxation system is an approximation of the Korteweg-de Vries (KdV) equation
\begin{align} \label{kdv}
    \eta_t + \eta \eta_x + \eta_{xxx} & = 0,
\end{align}
for which a hyperbolic approximation (referred to herein as KdVH) takes the form \cite{besse2022perfectly}
\begin{subequations}\label{Eq:kdvH}
\begin{align}
\partial_t u + u\partial_x u + \partial_x w &= 0 \;, \\
\tau \partial_t v &= (\partial_x v - w) \;, \\
\tau \partial_t w &= -(\partial_x u - v) \;,
\end{align}
\end{subequations}
with relaxation parameter $\tau>0$, where $u$ approximates $\eta$ and $v, w$ approximate $\eta_x, \eta_{xx}$ respectively. 
This system was originally proposed in \cite{besse2022perfectly} and has
recently been studied in \cite{biswas2025traveling}, where it was discretized with ImEx Runge-Kutta
methods in time and finite differences in space.
Therein it was shown that the time integration methods, under certain conditions
on the Butcher coefficients, provide an \emph{asymptotic-preserving} (AP) discretization \cite{jin2022asymptotic};
i.e., in the relaxation limit, the discrete system tends to a consistent discretization
of the KdV equation, { provided the initial conditions of the system are {\em well-prepared}:
\[
    \lim_{\tau\to0} v(x,0;\tau) = \lim_{\tau\to0}  \partial_x u(x,0;\tau), \quad 
    \lim_{\tau\to0} w(x,0;\tau) = \lim_{\tau\to0} \partial_x v(x,0;\tau)
\]
See Section \ref{sec:AP} of the present work for a more detailed discussion of well-prepared initial data.
}

Given that the stiff terms in \eqref{Eq:kdvH} are linear, it is natural to consider
alternatively the application of exponential integrators, wherein the effects of the
stiff linear terms are computed (in principle) exactly, through the use of the
matrix exponential.  One might even expect exponential methods to outperform
ImEx methods in this context, since the stiff terms are non-dissipative and
exponential methods seem more suited to preserving this property on the discrete level.
In the present work we study the application of exponential integrators for the
time discretization of \eqref{Eq:kdvH}. 
The main goal is to explore and analyze the asymptotic behavior of different classes of exponential methods for small $\tau$.
We consider one-step exponential methods
of two types: simple Lawson methods as well as exponential time differencing
(ETD) methods.  
We prove that ETD methods are asymptotic preserving for all components
of the system, including the solution variable $u$ and the auxiliary
(derivative-approximating) variables $v$ and $w$.
For Lawson methods, our analysis establishes that Lawson--Euler is
not asymptotic preserving in general. Our numerical experiments with
Lawson methods indicate that $u$ converges to the corresponding
numerical solution of the limiting KdV equation \eqref{kdv}
as $\tau\to0$, whereas $v$ and $w$ generally fail to satisfy
the equilibrium relations.
% We show that these methods are all asymptotic preserving
% with respect to the solution variable $u$, whereas only the ETD methods are
% AP with respect to the auxiliary (derivative-approximating) variables $v$ and $w$. 
% We show that, in general, Lawson methods are not asymptotic preserving; in practice with these methods we observe that the value of $u$ converges to the solution $\eta$ of the limit (KdV) equation \eqref{kdv} as $\tau \to 0$, but the auxiliary (derivative-approximating) variables $v$ and $w$ do not converge to the first and second spatial derivatives of $\eta$. By contrast, exponential time differencing (ETD) methods are AP for all components, including both the solution variable and the auxiliary variables.
The essential difference between these classes of methods seems to be rooted in the behavior of the functions $\phi_k$ that appear in exponential methods, with quite different behavior between the case $k=0$ versus $k>0$ in the limit $\tau\to0$ (see Section \ref{sec:AP}).

There exist a few studies of AP properties of exponential methods, all in the context of dissipative dynamics (generally in the context of kinetic equations, e.g. \cite{dimarco2011exponential, hu2019second, li2014exponential, li2014exponential11}).
To our knowledge, the present work is the first AP analysis of exponential methods applied to purely dispersive systems.

Although our comparison in this work involves only the specific KdVH system
with pseudospectral space discretization, we believe it provides an interesting
case study for the broader question of the relative efficiency of ImEx versus
exponential methods for hyperbolized systems and even for general PDEs.

The paper is organized as follows.  In Section \ref{sec:discretization} we describe the numerical discretization methods that will be studied and tested.  In Section \ref{sec:AP}, we analyze the AP property for a wide range of exponential methods applied to the KdVH system.
In Section \ref{sec:tests} we verify the results of the theoretical analysis through numerical tests.  We also
investigate the relative computational efficiency of ImEx and exponential
methods for this problem.  This is a delicate question and the answer  intimately depends
on the method used to compute the matrix exponential as well as the linear
solver used for the ImEx methods.

%-------------------------------%
\section{Numerical discretization}
\label{sec:discretization}
%-------------------------------%

In this section we describe the numerical discretizations that will be analyzed and tested.
In space, we use a standard Fourier pseudospectral
method, briefly reviewed in Section \ref{sec:space}.  In time, we are mainly interested in the behavior of exponential methods, which we review in Section \ref{sec:ERK}.  We will also make some comparisons with implicit-explicit Runge-Kutta methods, which we review in Section \ref{sec:ImEx}.
%=====================================================%
\subsection{Spatial discretization: pseudospectral Fourier collocation} \label{sec:space}
%=====================================================%
To study the AP property of the exponential methods, we formulate and evolve the KdV equation and the KdVH system in Fourier space. Consider the KdV equation \eqref{kdv} with periodic boundary conditions. Taking the Fourier transform and replacing the nonlinear convolution via the usual pseudospectral approximation, we obtain the KdV equation in Fourier space:
\begin{align}\label{Eq:kdv_Fourier_cont}
    \partial_t \hat{\ukdv}(\xi,t) & = -(i \xi)^3\hat{\ukdv}(\xi,t) - \frac{i \xi}{2} \mathcal{F} \left( \left[ \mathcal{F}^{-1}\left(\hat{\ukdv}(\xi,t)\right) \right]^2\right) \;.
\end{align}
Here $\xi$ is the wavenumber, $\hat{\ukdv}$ is the Fourier transform of $\eta$, and $\mathcal{F}(\cdot)$ is the Fourier transform operator. Similarly, we apply the Fourier transform to the KdVH system \eqref{Eq:kdvH} to obtain
\begin{align}\label{Eq:kdvh_Fourier_cont}
    \partial_t \hat{q}(\xi,t) & = L(\tau)\hat{q}(\xi,t) + N\left(\xi,\hat{q}(\xi,t)\right) \;,
\end{align}
where $\hat{q}(\xi,t) =
[\hat{u}(\xi,t), \ \hat{v}(\xi,t), \ \hat{w}(\xi,t)]^{T}$ and 

\begin{align} \label{Ltaudef}
 L(\tau) =\begin{bmatrix}
        0 & 0 & -i \xi \\
        0 & \frac{i\xi}{\tau} & -\frac{1}{\tau} \\
        -\frac{i\xi}{\tau} & \frac{1}{\tau} & 0
\end{bmatrix} \;, \quad \text{and} \ \ 
N\left(\xi,\hat{q}(\xi,t)\right) = 
\begin{bmatrix}
- \frac{i \xi}{2} \mathcal{F} \left( \left[ \mathcal{F}^{-1}\left(\hat{u}(\xi,t)\right) \right]^2\right) \\
0 \\
0
\end{bmatrix} \;.
\end{align}
% \begin{align} \label{Ltaudef}
% \hat{q}(\xi,t)= 
%             \begin{bmatrix}
%                 \hat{u}(\xi,t) \\ 
%                 \hat{v}(\xi,t) \\
%                 \hat{w}(\xi,t)
%             \end{bmatrix} \;, \ \
%  L(\tau) =\begin{bmatrix}
%         0 & 0 & -i \xi \\
%         0 & \frac{i\xi}{\tau} & -\frac{1}{\tau} \\
%         -\frac{i\xi}{\tau} & \frac{1}{\tau} & 0
% \end{bmatrix} \;, \quad \text{and} \ \ 
% N\left(\xi,\hat{q}(\xi,t)\right) = 
% \begin{bmatrix}
% - \frac{i \xi}{2} \mathcal{F} \left( \left[ \mathcal{F}^{-1}\left(\hat{u}(\xi,t)\right) \right]^2\right) \\
% 0 \\
% 0
% \end{bmatrix} \;.
% \end{align}
In Section \ref{sec:AP}, where we study the AP property with respect
to exponential time discretization, we will use \eqref{Eq:kdvh_Fourier_cont} as a starting point, leaving the wavenumber continuous.

To solve \eqref{Eq:kdv_Fourier_cont} or \eqref{Eq:kdvh_Fourier_cont} numerically, we discretize in $\xi$ on a grid with $m$ nodes, using Fourier collocation.  For
the KdV equation \eqref{Eq:kdv_Fourier_cont}, we obtain the following ordinary differential equations for the discrete Fourier modes:
\begin{align}
    \hat{\ukdv}_j'(t) & = -(i \xi_j)^3\hat{\ukdv}_{j}(t) - \frac{i \xi_j}{2} \left[\mathcal{F} \left( \left( \mathcal{F}^{-1}\left(\hat{\ukdv}(t)\right) \right)^2\right) \right]_j
\end{align}
where $\xi_j$ denote the wavenumbers corresponding to the chosen grid. We can write this semidiscretization in the form
\begin{align}\label{Eq:kdv_Fourier}
    \partial_t \hat{\ukdv}(t) & = \tilde{L} \hat{\ukdv}(t) + \tilde{N}(\hat{\ukdv}(t)) \;,
\end{align}
where 
$\hat{\ukdv}(t) = \begin{bmatrix}
 \hat{\ukdv}_1(t),
 \hat{\ukdv}_2(t),
 \ldots,
 \hat{\ukdv}_m(t)
 \end{bmatrix}^T$, $\tilde{L} = \operatorname{diag}\left(-(i\xi_1)^3, -(i\xi_2)^3, \ldots, -(i\xi_m)^3\right)$, and $\tilde{N}(\hat{\ukdv}(t))_j = 
 -\frac{i \xi_j}{2} \left[\mathcal{F} \left( \left( \mathcal{F}^{-1}\left(\hat{\ukdv}(t)\right) \right)^2\right) \right]_j$.
% \begin{subequations}\label{Eq:kdv_Fourier}
% \begin{align}
%     \partial_t \hat{\ukdv}(t) & = \tilde{L} \hat{\ukdv}(t) + \tilde{N}(\hat{\ukdv}(t)) \;,
% \end{align}
% where 
% \begin{align}
%  \hat{\ukdv}(t) = \begin{bmatrix}
%  \hat{\ukdv}_1(t)  \\
%  \hat{\ukdv}_2(t) \\
%  \vdots \\
%  \hat{\ukdv}_m(t)\\
%  \end{bmatrix} \;, \quad
%  \tilde{L} =\begin{bmatrix}
%     {-(i\xi_1)^3}  &  {}   & {}    &  {}  \\
%      {}    &  {-(i\xi_2)^3}  &  {}   & {}    \\
%      {}    &    {}   & {\ddots}  &    {}\\
%      {}    &   {}   &   {} & {-(i\xi_m)^3}
%  \end{bmatrix}\;, \quad \text{and} \quad
%  \tilde{N}(\hat{\ukdv}(t))_j = 
%  -\frac{i \xi_j}{2} \left[\mathcal{F} \left( \left( \mathcal{F}^{-1}\left(\hat{\ukdv}(t)\right) \right)^2\right) \right]_j.
% \end{align}
% \end{subequations}
We can semi-discretize the KdVH system \eqref{Eq:kdvh_Fourier_cont} in a similar way. Using the same equispaced grid with $m$ nodes
and applying a pseudospectral Fourier discretization yields
\begin{align} \label{semidiscretization}
    \partial_t \hat{\qh}(t) & = \tilde{\Lh} \hat{\qh}(t)  + \tilde{\Nh}(\hat{\qh}(t)),
\end{align}
where $\hat{\qh}(t) = \begin{bmatrix}
\hat{q}_1(t),
\hat{q}_2(t),
 \ldots,
\hat{q}_m(t)
 \end{bmatrix}^T$, $\tilde{\Lh} = \operatorname{diag}(L_1,L_2,\ldots,L_m)$, and $\tilde{\Nh}(\hat{\qh}(t))$ is the vector  
$\begin{bmatrix}
       \tilde{\Nh}(\hat{\qh}(t))_1,
       \tilde{\Nh}(\hat{\qh}(t))_2,
    \ldots,
      \tilde{\Nh}(\hat{\qh}(t))_m 
\end{bmatrix}^T$ with $\hat{q}_j(t) = \hat{q}(\xi_j,t)$ and
\begin{align} \label{Lblock}
%\hat{q}_j(t) = \begin{bmatrix}
%                    \hat{u}(\xi_j,t) \\
%                    \hat{v}(\xi_j,t) \\
%                    \hat{w}(\xi_j,t)
%                \end{bmatrix} \;, \quad
 L_j =\begin{bmatrix}
        0 & 0 & -i \xi_j \\
        0 & \frac{i\xi_j}{\tau} & -\frac{1}{\tau} \\
        -\frac{i\xi_j}{\tau} & \frac{1}{\tau} & 0
\end{bmatrix} \;, \quad \text{and} \quad
\tilde{\Nh}(\hat{\qh}(t))_j = 
\begin{bmatrix}
 -\frac{i \xi_j}{2} \left[\mathcal{F} \left( \left( \mathcal{F}^{-1}\left(\hat{u}(t)\right) \right)^2\right) \right]_j  \\
0 \\
0
\end{bmatrix} \;.
\end{align}

% \begin{subequations} \label{semidiscretization}
% \begin{align}
%     \partial_t \hat{\qh}(t) & = \tilde{\Lh} \hat{\qh}(t)  + \tilde{\Nh}(\hat{\qh}(t)),
% \end{align}
% where 
% \begin{align}
% \hat{\qh}(t) = \begin{bmatrix}
% \hat{q}_1(t) \\
% \hat{q}_2(t) \\
%  \vdots \\
% \hat{q}_m(t) \\
%  \end{bmatrix} \;, \quad
%  \tilde{\Lh} =
% \begin{bmatrix}
%     \boxed{L_1}  &  {}   & {}    &  {}  \\
%      {}    &  \boxed{L_2}  &  {}   & {}    \\
%      {}    &    {}   & \ddots  &    {}\\
%      {}    &   {}   &   {} & \boxed{L_m}
%  \end{bmatrix}\;, \quad \text{and} \quad
%  \tilde{\Nh}(\hat{\qh}(t)) = 
%  \begin{bmatrix}
%        \tilde{\Nh}(\hat{\qh}(t))_1  \\
%        \tilde{\Nh}(\hat{\qh}(t))_2 \\
%     \vdots \\
%       \tilde{\Nh}(\hat{\qh}(t))_m \\
% \end{bmatrix} \;.
% \end{align}
% \end{subequations}
% with $\hat{q}_j(t) = \hat{q}(\xi_j,t)$,
% \begin{align} \label{Lblock}
% %\hat{q}_j(t) = \begin{bmatrix}
% %                    \hat{u}(\xi_j,t) \\
% %                    \hat{v}(\xi_j,t) \\
% %                    \hat{w}(\xi_j,t)
% %                \end{bmatrix} \;, \quad
%  L_j =\begin{bmatrix}
%         0 & 0 & -i \xi_j \\
%         0 & \frac{i\xi_j}{\tau} & -\frac{1}{\tau} \\
%         -\frac{i\xi_j}{\tau} & \frac{1}{\tau} & 0
% \end{bmatrix} \;, \quad \text{and} \quad
% \tilde{\Nh}(\hat{\qh}(t))_j = 
% \begin{bmatrix}
%  -\frac{i \xi_j}{2} \left[\mathcal{F} \left( \left( \mathcal{F}^{-1}\left(\hat{u}(t)\right) \right)^2\right) \right]_j  \\
% 0 \\
% 0
% \end{bmatrix} \;.
% \end{align}

%===============================%
\subsection{Temporal discretization}
%===============================%
Motivated by the semi-discrete system \eqref{semidiscretization}, in this section we consider an abstract system of ODEs of the form
\begin{align}\label{Eq:semilienar_ode}
    \partial_t y(t) & = Ly(t) + N(t,y) \;,
\end{align}
%\giovanni{We have to fix the notation, and later on refer to Eq(6) for the 3x3 matrix L.}
where $y \in \mathbb{R}^d$, $L \in \mathbb{R}^{d \times d}$, and $N$ is a nonlinear function, $N: \mathbb{R} \times \mathbb{R}^{d}  \rightarrow \mathbb{R}^{d}$.
We are mainly interested
in evaluating the effectiveness of exponential methods for
the integration of \eqref{semidiscretization}.
We study two types of exponential methods: Lawson, or Integrating Factor (IF)
methods, and Exponential Time Differencing (ETD) methods.
We also describe some ImEx Runge-Kutta methods
that will be used for comparison later.

%-----------------------------------------%
\subsubsection{Exponential Runge-Kutta methods} \label{sec:ERK}
%-----------------------------------------%
Exponential RK methods provide a framework for the time integration of stiff differential equations of the form \eqref{Eq:semilienar_ode} \cite{ hochbruck2005explicit}.
An exponential RK method for \eqref{Eq:semilienar_ode} reads
\begin{subequations}\label{ExInt2}
\begin{align}
    Y_i &= \chi_{i}(\Delta t L) y_{n} + \Delta t \sum_{j=1}^{s} a_{ij}(\Delta t L) N(t_{n} + c_j \Delta t, Y_j), \quad i = 1, \dots, s, \\
    y_{n+1} &=\chi_0( \Delta t L ) y_{n} + \Delta t  \sum_{i=1}^{s} b_i(\Delta t L) N(t_{n} + c_i \Delta t, Y_i)\;.
\end{align}
\end{subequations}
The coefficient functions $\chi_0$, $\chi_i$, $a_{ij}$, and $b_{i}$ are constructed from exponential functions. 
%It turns out that the coefficients of the method \eqref{ExInt2} have to satisfy the conditions 
%$$
%\sum_{i = 1}^s b_i(z) = \phi_0(z)-1, \quad \sum_{i = 1}^s a_{ij}(0) = c_i, \, \, 1 \le i\le s.
%$$
Consistency of the method requires that
$\chi_0(0) = \chi_{i}(0) = 1$, and that 
$$
\sum_{i = 1}^s b_i(0) = 1, \quad \sum_{j = 1}^s a_{ij}(0) = c_i, \, \, 1 \le i\le s.
$$
Thus, if we set $L=0$ in \eqref{ExInt2}, we obtain a Runge-Kutta (RK) method with coefficients $b_i := b_i(0)$ and $a_{ij}:= a_{ij}(0)$; we refer to this as the \emph{underlying method}.
An exponential method can be specified, similarly to an RK method, by displaying its coefficients in an extended Butcher table:
\begin{equation}\label{Btab}
\begin{array}{c|cccc|c}
    c_1 & a_{11}(z) &a_{12}(z) & \quad\cdots\quad & a_{1s}(z) & \chi_{1}(z) \\
    c_2 & a_{21}(z) &a_{22}(z) & \quad\cdots\quad & a_{2s}(z) & \chi_{2}(z) \\
    \vdots & \vdots &  & \vdots & \vdots & \vdots \\
    c_s & a_{s1}(z) & a_{s2}(z) & \quad\cdots\quad & a_{ss}(z) & \chi_{s}(z) \\
    \hline
    & b_1(z) & b_2(z) & \quad\cdots\quad & b_s(z) & \chi_0(z)
\end{array} \;.
\end{equation}
% For $z\ne 0$, the coefficients of the method \eqref{Btab} must satisfy the consistency conditions \cite{hochbruck2005explicit}
% \begin{equation}\label{2cond}
% \sum_{i = 1}^s b_i(z) = \frac{\chi_0(z)-1}{z}, \quad \sum_{j = 1}^s a_{ij}(z) = \frac{\chi_{i}(z)-1}{z}, \, \, 1 \le i\le s.
% \end{equation}
For $z\ne 0$, the method \eqref{Btab} preserves equilibria at both
the stages and the time steps if its coefficients satisfy
\cite{hochbruck2005explicit}
\begin{equation}\label{2cond}
\sum_{i=1}^s b_i(z)=\frac{\chi_0(z)-1}{z}, \quad
\sum_{j=1}^s a_{ij}(z)=\frac{\chi_i(z)-1}{z},
\qquad 1\le i\le s.
\end{equation}
These conditions hold for the ETD methods considered here,
but not generally for Lawson methods. We only consider methods with 
\begin{equation}\label{1cond}
\chi_0(z) = e^{z}, \quad  \chi_{i}(z) = e^{c_i z},
\quad \quad c_i\ge 0, \quad \quad  1 \le i \le s.\\
\end{equation}
Note that as $z\to0$ the consistency conditions coincide with those of classical Runge-Kutta methods. The main motivation for considering exponential schemes of the form \eqref{ExInt2} is to construct explicit methods. To this end, we restrict our attention to schemes with 
$c_1 = 0$ which implies $\chi_1(z)=1$, and with coefficients $a_{i,j}=0$
for all  $1\le i\le j \le s$.
 
%We denote the internal stage approximation as $Y_i \approx y(t_n + c_i \Delta t)$ and the solution approximation as $y_n \approx y(t_n)$. 
%The coefficient $a_{ij}(0) = a_{ij}$, and $b_i(0) = b_i$, with $a_{ij}$ and $b_i$ real numbers, are the coefficients of a Runge–Kutta scheme, \cite{}.\\
%Within the family of exponential integrators, it is essential to distinguish between \emph{Integrating factor} (IF) and \emph{Exponential time differentiating} RK \emph{ETD (ETDRK)} methods.

Given a classical (underlying) RK method with coefficients $b_i$ and $a_{ij}$, the 
corresponding IF method \cite{lawson1967generalized} has
coefficients
\begin{equation}\label{coeffL}
a_{ij}(\Delta t L) = a_{ij}e^{(c_i-c_j)\Delta t L},\quad b_{i}(\Delta t L) = b_ie^{(1-c_i)\Delta t L}, \quad i,j = 1,...,s.
\end{equation}
The order of an IF method is the same as that of the underlying method.
The simplest IF method is the explicit Lawson-Euler method, given by
\begin{equation} \label{Lawson-Euler}
y_{n+1} =\chi_0( \Delta t L ) \left(y_{n} + \Delta t  N(t_{n}, y_n)\right).
\end{equation}

In order to define the class of ETD methods we recall the so-called $\phi$-functions:
$$
\phi_k(z) = \int_0^1 e^{(1-\theta)z}\frac{\theta^{k-1}}{(k-1)!}d \theta, \; z \in \mathbb{C}, \; k = 1,2,...
$$
which satisfy the recurrence relation (see e.g. \cite{hochbruck2010exponential})
$$
\phi_{0}(z) = e^z, \quad \phi_{k}(z) = \frac{\phi_{k-1}(z) - \frac{1}{(k-1)!}}{z}. 
$$
%The first few functions are given by
%\[
%    \phi_1(z) = \frac{e^z-1}{z}, \quad \phi_2(z) = \frac{e^z-1-z}{z^2}, \quad \phi_3(x) = \frac{e^z-1-z-z^2/2}{z^3}
%\]
and, in general
\begin{equation}
    \label{eq:phi_k}
    \phi_k(z) = \left( e^z-\sum_{j=0}^{k-1}\frac{z^j}{j!}\right)z^{-k} = \sum_{j=0}^\infty \frac{z^j}{(j+k)!}.
\end{equation}
Usually the weights $b_i(z)$ and coefficients $a_{ij}(z)$  of an ETD method are linear combinations of the functions $\phi_k(z)$ and $\phi_{k}(c_i z)$,
respectively \cite{hochbruck2005explicit}. Taking into account \eqref{1cond}, the consistency conditions \eqref{2cond} become
\begin{equation*}
\sum_{i = 1}^s b_i(z) = \phi_1(z), \quad \sum_{j = 1}^s a_{ij}(z) = c_i \phi_{1}(c_iz), \, \, 1 \le i\le s.
\end{equation*}
%$$
% b_{i}(\Delta t L) =\int_0^1 e^{(1-\theta)\Delta t L }p_i(\theta) d\theta, \quad 
%a_{ij}(\Delta t L) = \int_0^1 e^{(1-\theta)c_i\Delta t L }p_{ij}(\theta) d\theta, \quad i,j = 1,2,...,s
%$$
%where $p_i$ and $p_{ij}$ are polynomials \cite{cox2002exponential}.
To save space, we define $\phi_{k,i} = \phi_{k}(c_i \Delta t L)$.
The simplest ETD method is the explicit ETD-Euler method given by
\begin{equation}
y_{n+1} =\phi_0( \Delta t L ) y_{n} + \Delta t  \phi_1(\Delta t L)N(t_{n}, y_{n}).
\end{equation}
 
An exponential Runge–Kutta method is said to have {\em non-stiff order} $p$ if it has order $p$ when applied to sufficiently smooth non-stiff problems. 
An exponential Runge–Kutta method is said to have {\em stiff order} $q$ if it has order $q$ when applied to stiff semilinear problems \eqref{Eq:semilienar_ode}, uniformly with respect to the stiffness of the  linear operator.
Since we are dealing with a stiff semilinear system, we expect stiff order to be a relevant property in this work.
%\david{We are missing several ETD methods here. Abhijit, please add them.} \abhi{Among the six ETD methods used in this manuscript, namely ETD2RK, ETD3RK, ETD4RK, Hochbruck--Ostermann, and ETD5RKF, two are presented here, while the remaining methods are included in Appendix A. Should we move some of the methods here?}
The exponential methods considered in this work are listed below, and their coefficients are provided in the supplementary material.

 \begin{itemize}
    \item \textbf{IF (Lawson-type) methods} \cite{lawson1967generalized, hochbruck2005explicit}:
    \begin{enumerate}
        \item \textbf{Lawson--Euler:} The one-stage, first-order method \eqref{Lawson-Euler}.
    \item \textbf{Lawson2b:} A two-stage exponential IF method of the form \eqref{coeffL}
    based on the explicit Heun method.
     This method has non-stiff order 2 and stiff order 1.
    \item \textbf{Lawson4:} A four-stage exponential method of the form \eqref{coeffL}, based on the classical explicit fourth order Runge-Kutta method. This method has non-stiff order 4 and stiff order 1.
    \end{enumerate}
    \item \textbf{ETD RK methods}. The construction of various families of ETD RK methods can be found in  \cite{cox2002exponential, hochbruck2005explicit}. Here we report the ETD RK methods of orders one through four considered in this work.
    \begin{enumerate}
        \item \textbf{N{\o}rsett-Euler\cite{hochbruck2005explicit}:} A one-stage method with a non-stiff order 1 and a stiff order 1.
        \item \textbf{ETD2RK\cite{cox2002exponential, hochbruck2005explicit}:} A 2-stage method with a nonstiff order 2 and a stiff order 2.
        \item \textbf{ETD3RK\cite{cox2002exponential}:} A 3-stage method with a nonstiff order 3 and a stiff order 2.
        \item \textbf{ETD4RK\cite{cox2002exponential}:} A 4-stage method with a nonstiff order 4 and a stiff order 2.
        \item \textbf{Hochbruck–Ostermann\cite{hochbruck2005explicit}:} A 5-stage method with a nonstiff order 4 and a stiff order 4.
        %\item \SB{We should remove from the list this method if we do not use}: \textbf{ETD5RKF \cite{cox2002exponential}:} A 6-stage method with a non-stiff order of 5 and a stiff order of 1.
    \end{enumerate}
\end{itemize}

%-----------------------%
\subsubsection{ImEx RK methods} \label{sec:ImEx}
%-----------------------%
For comparison purposes, we also consider ImEx RK methods.
These methods can be very efficient when applied to problems of the form \eqref{Eq:semilienar_ode}, by integrating the stiff linear term implicitly and the non-stiff nonlinear term explicitly.
The methods take the form
\begin{subequations}
\label{eq:ImEx_mthd}
\begin{align}
    Y_i & = y_n + \Delta t \left(\sum_{j = 1}^{i-1} \tilde{a}_{ij} N(t_{n} + c_j \Delta t, Y_j) + \sum_{j = 1}^{i} a_{ij} L Y_j \right), \ i = 1,2, \ldots, s \;, \label{Eq:ImEx_mthd_stage_type_I}\\
    y_{n+1} & = y_n + \Delta t \left( \sum_{i = 1}^{s} \tilde{b}_{i} N(t_{n} + c_i \Delta t, Y_i) + \sum_{i = 1}^{s} b_{i} L Y_i \right)  \; \label{Eq:ImEx_mthd_sol_type_I}.
\end{align}
\end{subequations}
The coefficients pertaining to a specific method
are conveniently written using a double Butcher tableau \cite{boscarino2024implicit}
\begin{equation}\label{Mthd:ImEx-RK}
    \begin{array}{c|ccc}
      \tIcex & \tIAex \\
      \hline
      & \\[-1em]
      & \tIbex^T \\
    \end{array}
    \qquad
    \begin{array}{c|ccc}
      \tIcim & \tIAim \\
      \hline
      & \\[-1em]
      & \tIbim^T \\
    \end{array} \;,
\end{equation}
% \begin{equation}\label{Mthd:ImEx-RK}
%     \begin{array}{c|ccc}
%       \bm{\tIcex} & \tIAex \\
%       \hline
%       & \\[-1em]
%       & \tIbex^T \\
%     \end{array}
%     \qquad
%     \begin{array}{c|ccc}
%       \tIcim & \tIAim \\
%       \hline
%       & \\[-1em]
%       & \tIbim^T \\
%     \end{array} \;,
%   \end{equation}
where the strictly lower-triangular matrix $\tIAex = (\tilde{a}_{ij}) \in \mathbb{R}^{s \times s}$ represents the explicit part, the lower-triangular matrix $A = (a_{ij}) \in \mathbb{R}^{s \times s}$ represents the diagonally implicit part, and the vectors $\tIcex$, $\tIbex$, $\tIcim$, and $\tIbim$ are in $\mathbb{R}^{s}$. 
The methods that we use in this work  are listed, with
some of their properties, in Table \ref{tbl:imex}.
These are taken from a variety of sources
\cite{ascher1997implicit, boscarino2024asymptotic, kennedy2003additive, kennedy2019higher,pareschi2005implicit}.

\begin{table}[!tbp]
\centering
\begin{tabular}{c|cccccc} \hline
Name & Type & $S_I$ & $S_E$ & Order & SA & FSAL \\ \hline
AGSA(3,4,2)      & I & 3 & 4 & 2 & Yes & Yes \\
%SSP2-ImEx(3,3,2) & I & 3 & 3 & 2 & Yes & No \\
%SSP3-ImEx(3,4,3) & I & 3 & 4 & 3 & No & No \\
ARS(4,4,3)      & II & 4 & 4 & 3 & Yes & Yes \\
ARK3(2)4L[2]SA  & II & 4 & 4 & 3 & Yes & No \\
ARK4(3)6L[2]SA & II & 6 & 6 & 4 & Yes & No \\
ARK4(3)7L[2]SA & II & 7 & 7 & 4 & Yes & No \\
\hline
\end{tabular}
% \caption{Properties of ImEx methods included in this work.  Here $S_I$ and $S_E$ denote the number of stages of the implicit and explicit part, respectively.  See the text for explanation of the other terms.} \label{tbl:imex}
\caption{Properties of the ImEx methods included in this work. Here $S_I$ and $S_E$ denote the number of stages of the implicit and explicit parts, respectively; SA denotes stiffly accurate and FSAL denotes first-same-as-last. See \cite{boscarino2024implicit} for further details.}
\label{tbl:imex}
\end{table}

%==============================================%
\section{Asymptotic preserving properties of exponential integrators}\label{sec:AP}
%==============================================%

Having reviewed two classes of exponential methods above,
in this section we study for each class 
whether it preserves
the asymptotic behavior of the KdVH system.  
Recall that as $\tau \to 0$ the KdVH system   formally
becomes equivalent to the KdV equation \eqref{kdv}.  Correspondingly,
here we consider various time discretizations of \eqref{Eq:kdvh_Fourier_cont}, and analyze whether they converge to a consistent discretization of \eqref{Eq:kdv_Fourier_cont} as $\tau \to 0$ 
(keeping $\Delta t$ fixed).

Throughout this work, where required we assume that the solution is sufficiently differentiable with respect to $x$, $t$, and $\tau$.

\subsection{The matrix exponential for the KdVH system}
In this section we determine important properties of the exponential of the matrix $L(\tau)$ in \eqref{Ltaudef}, which will be required in our analysis.
%Clearly, this is just the block matrix of the exponentials of $L_j$ in \eqref{Lblock}.
 
Computing the  characteristic polynomial of $L(\tau)$ 
and multiplying by $\tau^2$ we obtain
\begin{align*}
P(\lambda) = \tau^2\lambda^3 - i \xi \tau\lambda^2 + (\xi^2 \tau + 1)\lambda - i\xi^3. 
\end{align*}
For small $\tau$, the spectrum of $L(\tau)$ consists of one slow mode and two fast modes, with eigenvalues
% \begin{align}
% \lambda_{s} = i\xi^3 + \mathcal{O}(\tau)
% \end{align}
% and 
% \begin{align}\label{Climit}
% \lambda_{\pm} = i \tau^{-1}\omega_{\pm} + \frac{i\xi^3 - \xi^2 i\omega_{\pm}}{-3i \omega_{\pm}^3 + 2 \xi \omega_{\pm} + 1} + \mathcal{O}(\tau), 
% \end{align}
\begin{equation}\label{Climit}
\lambda_s=i\xi^3+\mathcal{O}(\tau),
\qquad
\lambda_{\pm}
=i\tau^{-1}\omega_{\pm}
+
\frac{i\xi^2(\omega_{\pm}-\xi)}
{3\omega_{\pm}^2-2\xi\omega_{\pm}-1}
+\mathcal{O}(\tau) \;,
\end{equation}
%\abhi{It looks like there is a typo in the formula for $\lambda_{\pm}$: the term $-3i\omega_{\pm}^3$ in the denominator should be $-3\omega_{\pm}^2$.}
respectively,
where $\omega_{\pm} = \frac{\xi \pm \sqrt{\xi^2 + 4}}{2}$.
Thus, the matrix $L(\tau)$ is diagonalizable with $3$ distinct eigenvalues.
% Note that
% \begin{equation}\label{wpwm}
% \omega_+ \omega_- = -1, \quad \omega_+ + \omega_- = \xi.
% \end{equation}
% Furthermore, from the expansion of $\lambda_{\pm}$ in \eqref{Climit},
% setting $\mu = i \omega_+/\tau$ and $\nu = i \omega_-/\tau$, we obtain,
% using \eqref{wpwm},
% \begin{equation}\label{semp}
% \mu \nu = 1/\tau^2, \quad \mu + \nu = i\xi/\tau.
% \end{equation}

The key element in our analysis of the AP property is the
behavior of the functions $\phi_k(\Delta t L(\tau))$ as $\tau \to 0$.  In what follows we explicitly compute these functions.  For $k>0$, we show in Lemmas \ref{Lemma:phi_i_fun_eval}-\ref{lem:phiproduct} that these functions eliminate the rapidly-oscillating components of the hyperbolic relaxation and preserve only the component relevant to the original KdV equation.  On the other hand, for $k=0$, we show that $\phi_k(\Delta t L(\tau))$ retains the fast oscillations when $\tau$ is small.

The Lagrange-Sylvester formula (see e.g. \cite{higham2008functions}) says that if $A$ is a
diagonalizable matrix with distinct eigenvalues and $f$ is an analytic function, then
\begin{align} \label{Eq:LS_formula}
  f(A)& =\sum _{i=1}^{n}f(\lambda _{i})~A_{i}~, &
    \textup{where} \quad \quad \quad
    A_{i}\equiv \prod_{\substack{j=1 \\ j \neq i}}^{n}{\frac {1}{\lambda _{i}-\lambda _{j}}}\left(A-\lambda _{j}I\right).
\end{align}
are the Frobenius covariants of $A$.
Applying this formula gives
\begin{align}\label{Eq:phi0_fun}
\phi_0(\Delta t L(\tau)) := e^{\Delta t L(\tau)} &= e^{\Delta t \lambda_s} P_{s} + e^{\Delta t \lambda_{+}} P_+ + e^{\Delta t \lambda_{-}} P_-,
\end{align} 
where the Frobenius covariants of $L(\tau)$ are
\begin{subequations}\label{Eq:L_cov_mat}
\begin{align}
P_s &= \frac{(L(\tau) - \lambda_{+} I)(L(\tau) - \lambda_{-} I)}{(\lambda_s - \lambda_{+})(\lambda_s - \lambda_{-})}, \label{P_s} \\
P_+ &= \frac{(L(\tau) - \lambda_s I)(L(\tau) - \lambda_{-} I)}{(\lambda_{+} - \lambda_s)(\lambda_{+} - \lambda_{-})}, \\
P_- &= \frac{(L(\tau) - \lambda_s I)(L(\tau) - \lambda_{+} I)}{(\lambda_{-} - \lambda_s)(\lambda_{-} - \lambda_{+})}.
\end{align}
\end{subequations}
We are interested in computing the product
of $\phi_0$ with certain vectors in the limit
$\tau \to 0$.  The following intermediate results will be useful in this regard.
\begin{lemma}\label{lem:covariants}
Let $P_s$, $P_+$, and $P_-$ denote the Frobenius covariants
associated with the eigenvalues $\lambda_s$, $\lambda_+$, and
$\lambda_-$ of $L(\tau)$, respectively, as defined in
\eqref{Eq:L_cov_mat}. For each fixed wavenumber $\xi$, the limits as
$\tau\to0$ are
\begin{equation}
\lim_{\tau\to0}P_s
=
\begin{bmatrix}
1 & 0 & 0\\
i\xi & 0 & 0\\
(i\xi)^2 & 0 & 0
\end{bmatrix},
\quad
\lim_{\tau\to0}P_\pm
=
\begin{bmatrix}
0 & 0 & 0\\[1mm]
-\dfrac{2i\xi}{d_\pm}
&
\dfrac{\xi^2\pm\xi s+2}{d_\pm}
&
\pm\dfrac{i}{s}\\[3mm]
\dfrac{\xi(\xi\mp s)}{d_\pm}
&
\mp\dfrac{i}{s}
&
\dfrac{2}{d_\pm}
\end{bmatrix},
\end{equation}
where $s=\sqrt{\xi^2+4}$ and $d_\pm=s(s\pm\xi)$. In particular, for
$r(\xi)=[1,i\xi,(i\xi)^2]^T$, we have
\[
\lim_{\tau\to0}P_+r(\xi)
=
\lim_{\tau\to0}P_-r(\xi)
=0.
\]
\end{lemma}
\begin{proof}
Let $M(\tau):=\tau L(\tau)$. 
Since $M$ is a scalar multiple of $L(\tau)$, it has the same Frobenius covariants
$P_s$, $P_+$, and $P_-$, which are associated with the scaled eigenvalues
\[
\alpha_s(\tau)=\tau\lambda_s(\tau),
\qquad
\alpha_\pm(\tau)=\tau\lambda_\pm(\tau).
\]
From the eigenvalue expansions in \cref{Climit}, we have
\[
\lim_{\tau\to0}\alpha_s(\tau) = 0,
\qquad
\lim_{\tau\to0}\alpha_\pm(\tau) = i\omega_\pm \;.
\]

Moreover, in the limit $\tau \to 0$, 
$$M(\tau)
=
\begin{bmatrix}
0 & 0 & -i\xi\tau\\
0 & i\xi & -1\\
-i\xi & 1 & 0
\end{bmatrix}
\longrightarrow
M_0:=
\begin{bmatrix}
0 & 0 & 0\\
0 & i\xi & -1\\
-i\xi & 1 & 0
\end{bmatrix}.$$

From the Frobenius-covariant formula we have
$P_s
=\frac{(M-\alpha_+I)(M-\alpha_-I)}
     {(\alpha_s-\alpha_+)(\alpha_s-\alpha_-)}$ and since the denominator converges to $(-i\omega_+)(-i\omega_-)=1$ in the limit $\tau \to 0$, the $\lim_{\tau\to0}P_s$ exists and is given by
\[
\lim_{\tau\to0}P_s
=
\frac{(M_0-i\omega_+I)(M_0-i\omega_-I)}
     {(-i\omega_+)(-i\omega_-)}
     = (M_0-i\omega_+I)(M_0-i\omega_-I).
\]
Furthermore,
using $\omega_++\omega_-=\xi$ and $\omega_+\omega_-=-1$,
\[
(M_0-i\omega_+I)(M_0-i\omega_-I)
=
M_0^2-i\xi M_0+I
=
\begin{bmatrix}
1 & 0 & 0\\
i\xi & 0 & 0\\
(i\xi)^2 & 0 & 0
\end{bmatrix} \;,
\]
where the last equality is obtained by direct calculation, which proves the stated limit for $P_s$.

For $P_\pm$, the corresponding denominator converges to
$(i\omega_\pm)(i\omega_\pm-i\omega_\mp)$, which is nonzero since
$\omega_+\omega_-=-1$ and
$\omega_+-\omega_-=s=\sqrt{\xi^2+4}>0$. Thus the limits exist, and direct calculation gives
\[
\lim_{\tau\to0}P_\pm
=
\frac{M_0\bigl(M_0-i\omega_\mp I\bigr)}
     {(i\omega_\pm)\bigl(i\omega_\pm-i\omega_\mp\bigr)} = \begin{bmatrix}
0 & 0 & 0\\[1mm]
-\dfrac{2i\xi}{d_\pm}
&
\dfrac{\xi^2\pm\xi s+2}{d_\pm}
&
\pm\dfrac{i}{s}\\[3mm]
\dfrac{\xi(\xi\mp s)}{d_\pm}
&
\mp\dfrac{i}{s}
&
\dfrac{2}{d_\pm}
\end{bmatrix} \;,
\]
where we used $\omega_+-\omega_-=s=\sqrt{\xi^2+4}$ and where $d_\pm=s(s\pm\xi)$.

Finally, observe that $M_0r(\xi)=0$. Hence,
\[
\left(\lim_{\tau\to0}P_\pm\right)r(\xi)
=
\frac{M_0\bigl(M_0-i\omega_\mp I\bigr)r(\xi)}
     {(i\omega_\pm)\bigl(i\omega_\pm-i\omega_\mp\bigr)}
=0.
\]
This completes the proof.
\end{proof}
%\color{black}

The above lemma shows that the rapidly-oscillating parts of $\phi_0$ are orthogonal to well-prepared data in the limit $\tau\to0$.  However, they are not orthogonal to the nonlinear term, so rapid oscillations may be excited in schemes (like Lawson-Euler) that use $\phi_0$.

In order to prove the AP property of exponential integrators, we need the evaluation of higher-order $\phi$-functions of the matrix. 
\begin{lemma}\label{Lemma:phi_i_fun_eval}
    The higher-order $\phi$-functions evaluated at $\Delta t L(\tau)$ via the Lagrange--Sylvester formula are given by
    \begin{align}
        \phi_k(\Delta t L(\tau)) &= \phi_k(\Delta t\lambda_s) P_{s} + \phi_k(\Delta t \lambda_{+}) P_+ + \phi_k(\Delta t \lambda_{-}) P_-, \quad \text{for } k \geq 1,
    \end{align}
    where the Frobenius covariant matrices are given in \eqref{Eq:L_cov_mat}.
\end{lemma}
\begin{proof}
    By comparing the last formula in \eqref{eq:phi_k}
    term-by-term with the series for $\exp^{|z|}$, one
    immediately finds that $\phi_k$ is entire, so the Sylvester-Lagrange formula applies, and it gives
    the formula above.
\end{proof}

% \begin{lemma} \label{lem:phitau}
% For all $k\ge 1$, $\lim_{\tau \to 0} \phi_k(\Delta t \lambda_\pm) = 0.$ 
% \end{lemma}
\begin{lemma}\label{lem:phitau}
For each fixed wavenumber $\xi$ and fixed $\Delta t>0$,
$$\lim_{\tau\to0}\phi_k(\Delta t\lambda_\pm)=0,
\qquad k\ge1.$$
\end{lemma}
\begin{proof}
For $k=1$, we have
\begin{align} \label{phi1dtlam}
\phi_1(\Delta t \lambda_{\pm}) =  \frac{e^{\Delta t \lambda_{\pm}}- 1}{\lambda_{\pm} \Delta t},
\end{align}
where
\begin{align*}
\lambda_{\pm} = i \tau^{-1}\omega_{\pm} + \lambda^0_\pm +  \mathcal{O}(\tau), \ \text{with} \ \lambda^0_\pm = \frac{i\xi^2\left(\omega_{\pm} - \xi \right)}{3 \omega_{\pm}^2 - 2 \xi \omega_{\pm} - 1}
= \pm \frac{i\xi^2\left(\sqrt{\xi^2 + 4} \mp \xi\right)}{\left(\xi^2 \pm\xi\sqrt{\xi^2 + 4}  + 4\right)} \;.
\end{align*}
% where
% \[
% \lambda^0_\pm = \frac{i\xi^2\left(\omega_{\pm} - \xi \right)}{3 \omega_{\pm}^2 - 2 \xi \omega_{\pm} - 1}
% = \pm \frac{i\xi^2\left(\sqrt{\xi^2 + 4} \mp \xi\right)}{\left(\xi^2 \pm\xi\sqrt{\xi^2 + 4}  + 4\right)}
% \]
The modulus of the numerator of 
\eqref{phi1dtlam} is bounded, while the denominator is unbounded as $\tau \to 0$.

Next, assume by induction that the result holds for a given value of $k$.  We have
\begin{align*}
    \phi_{k+1} (\Delta t\lambda_\pm) & = \frac{\phi_k(\Delta t\lambda_\pm) - \frac{1}{k!}}{\Delta t\lambda_\pm}.
\end{align*}
Again, the numerator is bounded by two while the denominator grows without bound as $\tau \to 0$.
\end{proof}

Finally, we can compute the desired products.

\begin{lemma} \label{lem:phiproduct}
For each fixed wavenumber $\xi$, fixed $\Delta t>0$, and all $k \ge 1$, we have
    \begin{align}\label{Eq:Phi_k_term_lim_general}
        \lim_{\tau \to 0} \left(\phi_{k}(\Delta t L(\tau))e_1\right) 
                 & = \phi_{k}(i\xi^3 \Delta t) \begin{bmatrix}
                        1 \\
                        i \xi \\
                        (i \xi)^2
                    \end{bmatrix}  \;.
    \end{align}
\end{lemma}
\begin{proof}
We have $\lim_{\tau \to 0} \left(\phi_{k}(\Delta t L(\tau))e_1\right)$  
    \begin{align*} 
     & \quad = \lim_{\tau \to 0} \left(\phi_{k}(\Delta t \lambda_s) P_{s}e_1\right)
      + \lim_{\tau \to 0} \left(\phi_{k}(\Delta t \lambda_{+}) P_{+} e_1\right)
      + \lim_{\tau \to 0} \left(\phi_{k}(\Delta t \lambda_{-}) P_{-} e_1\right) \\
       & \quad = \left( \lim_{\tau \to 0} \phi_{k}(\Delta t \lambda_s) \cdot \lim_{\tau \to 0} P_{s}
      + \lim_{\tau \to 0} \phi_{k}(\Delta t \lambda_{+}) \cdot \lim_{\tau \to 0} P_{+}
      + \lim_{\tau \to 0} \phi_{k}(\Delta t \lambda_{-}) \cdot \lim_{\tau \to 0} P_{-}\right)e_1\;.
    \end{align*}
    The second and third terms vanish due to Lemma \ref{lem:phitau}.
    For the first term, we directly compute
    \begin{align*}
    \lim_{\tau \to 0} \left(\phi_{k}(\Delta t \lambda_s) P_{s}e_1\right) & = \lim_{\tau \to 0} \left(\phi_{k}(\Delta t \lambda_s) \right) \cdot \lim_{\tau \to 0} (P_{s}e_1) 
    = \phi_{k}(i\xi^3 \Delta t) \begin{bmatrix}
                        1 \\
                        i \xi \\
                        (i \xi)^2
                    \end{bmatrix}  \;.
    \end{align*}
\end{proof}
%-------------------------------------%
\subsection{Well-prepared initial data}
%-------------------------------------%
Whereas the KdV equation requires an initial condition for only one function $\eta(x,0)$, the KdVH system requires initial conditions for three functions $u \approx \eta$, $v \approx \eta_x$, and $w \approx \eta_{xx}$. It is important to consider whether the initial data are chosen in a way that is consistent with the relaxation limit $\tau \to 0$.
We say that the initial data for \eqref{Eq:kdvH} are {\em consistent} (or \emph{well-prepared to order zero)} if, as $\tau \to 0$, they satisfy
\begin{equation}\label{consistent}
v(x,0)=\partial_x u(x,0), \quad  w(x,0)=\partial_x v(x,0),
\end{equation}
 i.e. if $v^\tau(x,0)=\partial_x u^\tau(x,0)+O(\tau)$ and  $w^\tau(x,0)=\partial_x v^\tau(x,0)+O(\tau)$.
 This concept can be generalized to higher order in $\tau$ (see \cite{boscarino2024implicit}, Section 2.1.1). In Appendix~\ref{sec:well_prepared_order_1} we show how to impose well-prepared initial condition up to order 1 in $\tau$ for system \eqref{Eq:kdvH}.

For the classical hyperbolic relaxation approach of Jin \& Xin, 
the notion of AP does not require the initial condition to be well prepared. This is because, in that setting, the system dynamics are dissipative, so the relaxed manifold is a dynamical attractor \cite{boscarino2024asymptotic}.  If the initial condition is not well-prepared, an initial layer will appear, but will decay exponentially on a time scale $O(\tau)$. 
Our system \eqref{Eq:kdvH} is not dissipative, so that if the initial condition is not well-prepared the solution of the system will not converge to the solution of 
KdV equation.
Accordingly, the definition of AP for system \eqref{Eq:kdvh_Fourier_cont} reads as follows.
\begin{defn}
Consider the system \eqref{Eq:kdvh_Fourier_cont} with well-prepared initial data.
A consistent time discretization method is asymptotic preserving (AP)
if, as $\tau\to0$ with $\Delta t$ fixed, it reduces to a consistent
time discretization of \eqref{Eq:kdv_Fourier_cont}, with the auxiliary
components satisfying the equilibrium relations in the limit.
\end{defn}
Note that this definition in general does not imply that the scheme preserves the order of
accuracy in time in the stiff limit $\tau \to 0$. In the latter case the scheme is said to be
{\em asymptotically accurate} (AA) \cite{boscarino2024implicit}. Accordingly, the definition of AA for system \eqref{Eq:kdvh_Fourier_cont} reads as follows.
\begin{defn}
An exponential Runge-Kutta method of order $p$ applied to the KdVH system
\eqref{Eq:kdvh_Fourier_cont} with well-prepared initial data,
is asymptotically accurate (AA) if it converges to an exponential scheme of the same order applied to the limit KdV equation
\eqref{Eq:kdv_Fourier_cont}.
\end{defn}

% Accordingly, the definition of AP for system \eqref{Eq:kdvH} reads as follows.
% \begin{defn}%{\bf Asymptotic preserving} (\cite{jin2022asymptotic, jin1995relaxation, Jin2012}).
% Consider the system \eqref{Eq:kdvH} with well-prepared initial data.
% A consistent time discretization method
% is asymptotic preserving (AP) for this system if in the limit $\tau\to 0$ it becomes 
% (independently of the step size $\Delta t$) a consistent time discretization method for the reduced system \eqref{Eq:kdv_Fourier}.
% \end{defn}
% Note that this definition in general does not imply that the scheme preserves the order of
% accuracy in time in the stiff limit $\tau \to 0$. In the latter case the scheme is said to be
% {\em asymptotically accurate} (AA) \cite{boscarino2024implicit}. Accordingly, the definition of AA for system \eqref{Eq:kdvH} reads as follows.
% \begin{defn}
% An exponential Runge-Kutta method of order $p$ applied to the KdVH system 
%  \eqref{Eq:kdvh_Fourier_cont} with well-prepared initial data, 
% is asymptotically accurate (AA) if it converges to an exponential scheme of the same order applied to the limit KdV equation 
% \eqref{kdv}.
% \end{defn}

Furthermore, a numerical scheme is said to be \emph{uniformly accurate} (UA) if its order of accuracy is maintained uniformly for all values of the parameter $\tau$.
If $\tau$ is very small, it is usually sufficient to consider just consistent initial data. However, for a method which is of order $p$ in time and is AA, there could be a considerable degradation of accuracy if $\Delta t \gg \tau \gg \Delta t^p$. In such a case the accuracy can be drastically improved by adopting higher order well-prepared initial data.
In Figure \ref{fig:well-prepared}, we illustrate the effect of using data that is or is not well prepared in the context of a single solitary wave of speed $1.2$ propagated up to time $t=0.01$, with $\tau=10^{-4}$.  In the first subfigure, the data are well prepared to high accuracy, which is achieved by using Petviashvili's method as described later.  In the second subfigure, we simply set $v=w=0$ initially.  This leads to the generation of spurious oscillations, associated with the other modes introduced by hyperbolization.  In the third subfigure, we take 
consistent initial data, which is %approximately (but not precisely) 
well-prepared to order zero,
by setting $v=u_x$, $w=u_{xx}$ initially.  In this case, the oscillations that are generated are smaller, by approximately a factor of $\tau$.
In the last panel we adopt an initial condition which is well-prepared up to first order (see Appendix~\ref{sec:well_prepared_order_1}), which shows much smaller spurious oscillations.
 \begin{figure}
     \centering
     \begin{subfigure}{0.43\textwidth}
         \centering
         \includegraphics[width=\textwidth]{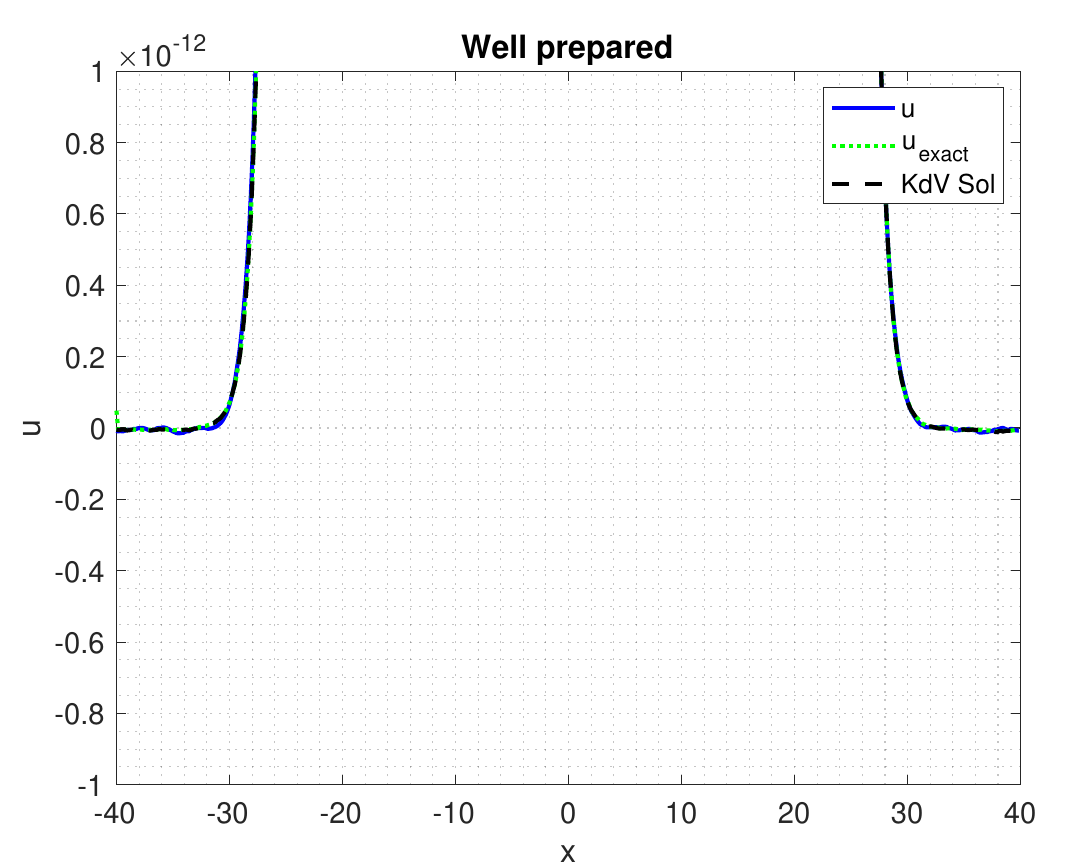}
     \end{subfigure}
     \hspace{0.00\textwidth} % Horizontal space
     \begin{subfigure}{0.43\textwidth}
         \centering
         \includegraphics[width=\textwidth]{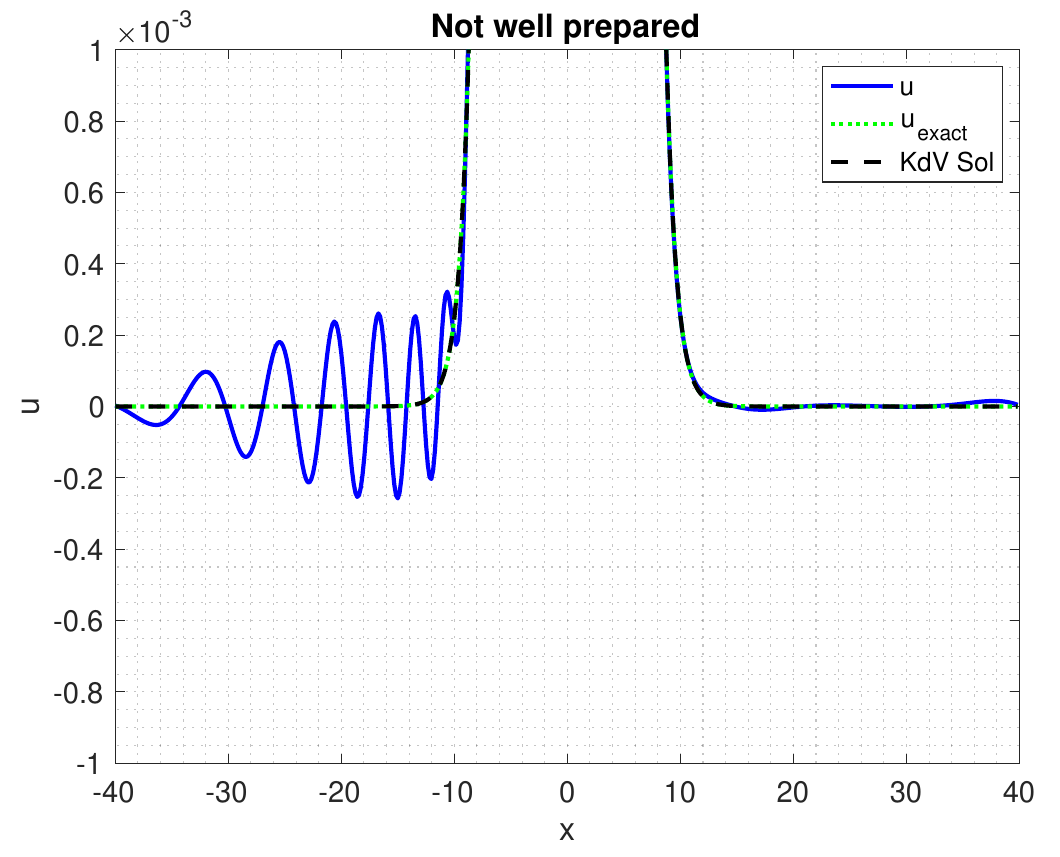}
     \end{subfigure}
%     \vspace{0.4cm} % Vertical spacing between rows
%     % Second Row
     \begin{subfigure}{0.43\textwidth}
         \centering
         \includegraphics[width=\textwidth]{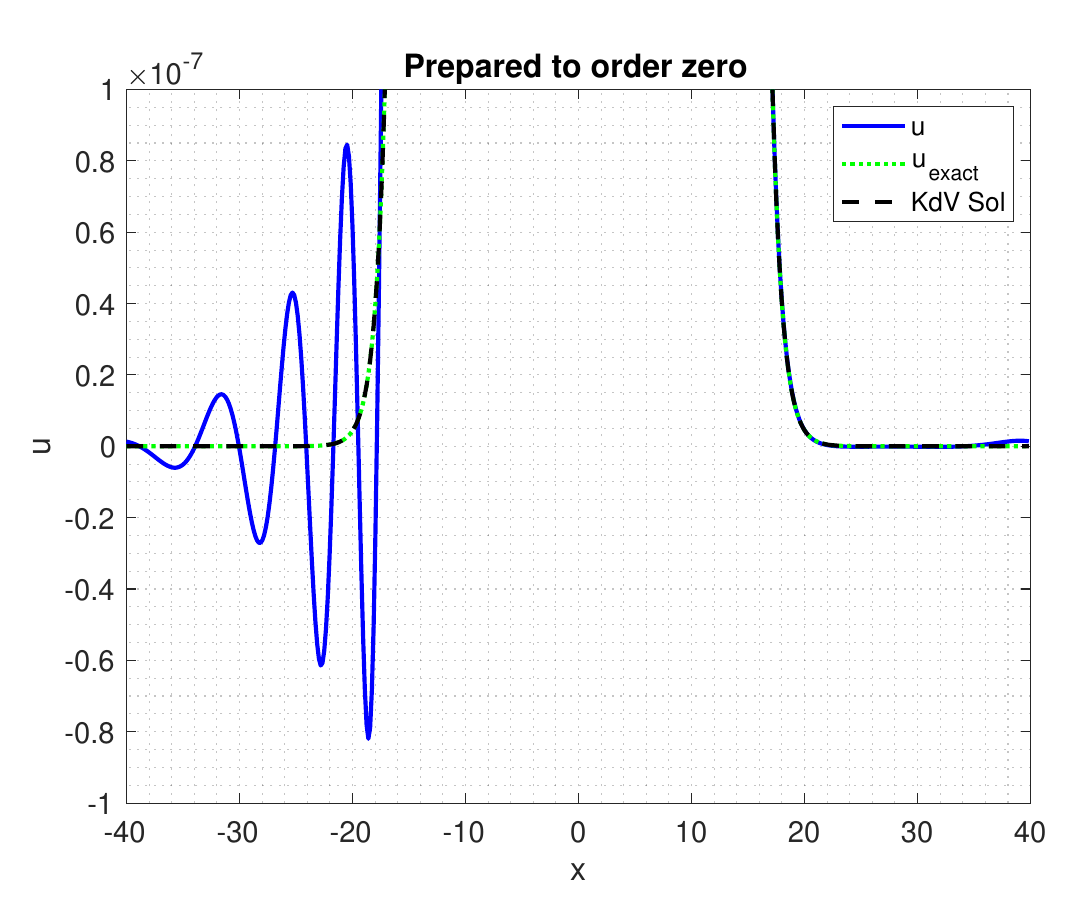}
     \end{subfigure}
     \begin{subfigure}{0.43\textwidth}
         \centering
         \includegraphics[width=\textwidth]{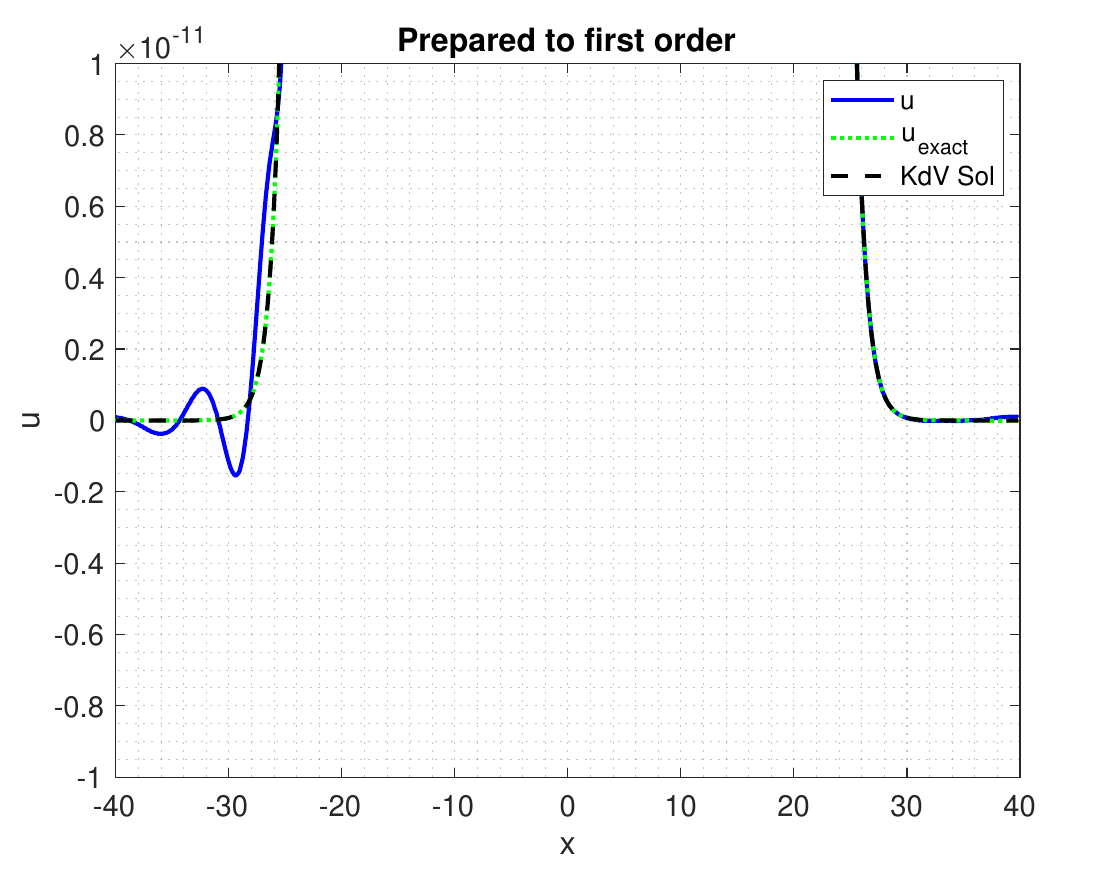}
     \end{subfigure}
     \hspace{0.05\textwidth} % Horizontal space
     \caption{Effect of well-prepared and not well-prepared initial data on the propagation of a solitary wave up to $t=0.01$ with $\tau=10^{-4}$. Here, $u$ denotes the computed $u$-component of the KdVH solution, $u_{\mathrm{exact}}$ denotes the corresponding reference traveling-wave solution obtained using Petviashvili's method, and $\mathrm{KdV \ Sol}$ denotes the numerical solution of the  KdV equation. Note the different vertical scales.}
     \label{fig:well-prepared}
 \end{figure}
%-----------------------------------------------%
\subsection{Analysis of asymptotic preservation}
%----------------------------------------------%
We are finally ready to state and prove the main results regarding the asymptotic preserving (AP) property for exponential discretizations of the KdVH system. Specifically, in this section, we prove that Lawson exponential methods are generally not AP, whereas ETD exponential methods are AP.

We begin by observing that a necessary condition for the AP property for the  exponential integrator applied to system \eqref{Eq:kdvH} is that the initial data are consistent. Hence, when $\tau \to 0$ we expect that the exponential scheme at each step $n$, solves  the KdV equation \eqref{kdv} for the $u$-component and projects the data $v$ and $w$ into the equilibrium manifold, namely, $v^n=\partial_x u^n$ and  $w^n=\partial_x v^n$. Therefore, in the following Theorems, we assume that the initial data $\qhat_0 = [\uh_0, \vh_0, \wh_0]^\top$, are consistent, namely, in the Fourier space, the consistent initial data \eqref{consistent} can be rewritten as

\begin{align}\label{Eq:well_prep}
    \hat{v}(\xi,0) = i\xi \hat{u}(\xi,0) \quad \text{and} \quad \hat{w}(\xi,0) = (i \xi)^2 \hat{u}(\xi,0).
\end{align}

Under this assumption, first we prove that the 1st order Lawson-Euler exponential integrator is not AP.
\begin{thm}\label{Th:AP_Lawson_Euler}
    The Lawson–Euler method applied to the system \eqref{Eq:kdvh_Fourier_cont} with consistent initial data \eqref{Eq:well_prep} 
    is not AP. %asymptotic-preserving for the $u$-component, but not AP for the $v$- and $w$-components.
\end{thm}

\begin{proof}
Let $\qhat^n=[\uh^n,\vh^n,\wh^n]^\top$ denote the approximation of $\qhat(\xi,t_n)$ in Fourier space. The Lawson--Euler method advances the solution according to
\begin{align}
    \qhat^{n+1}
    =
    e^{\Delta t L(\tau)}\qhat^{n}
    +
    \Delta t\,\phi_{0}(\Delta t L(\tau))N(\qhat^{n}) \;.
\end{align}
We show that the method is not AP by considering a single step from $t_0$ to $t_1$. Let the initial data be well prepared,
$\qhat^0=[\uh^0,\vh^0,\wh^0]^\top$, with
$\vh^0=i\xi\uh^0$ and $\wh^0=(i\xi)^2\uh^0$.
A necessary condition for asymptotic preservation is that, for fixed $\Delta t$,
$\vh^1-i\xi\uh^1\to0$ and
$\wh^1-(i\xi)^2\uh^1\to0$ as $\tau\to0$.

We first consider the linear contribution. Using the spectral decomposition of $L(\tau)$,
\begin{align}
    e^{\Delta tL(\tau)}\qhat^0
    &=
    e^{\Delta t\lambda_s}P_s\qhat^0
    +e^{\Delta t\lambda_+}P_+\qhat^0
    +e^{\Delta t\lambda_-}P_-\qhat^0 .
\end{align}
By Lemma~\ref{lem:covariants}, for well-prepared initial data,
$P_\pm\qhat^0\to0$ as $\tau\to0$, whereas
$e^{\Delta t\lambda_\pm}$ remain uniformly bounded. Hence
\begin{align}\label{Eq:LE_1st_term_lim}
    e^{\Delta tL(\tau)}\qhat^0
    =
    e^{i\xi^3\Delta t}
    \begin{bmatrix}
        \uh^0\\
        i\xi\uh^0\\
        (i\xi)^2\uh^0
    \end{bmatrix}
    +o(1),
    \qquad \tau\to0 .
\end{align}

We next consider the nonlinear contribution. Writing $N^0:=N(\xi,\uh^0)$ and using the Lemma~\ref{lem:covariants} we get
\begin{align*}
\phi_0(\Delta tL(\tau))N^0
&=
\left(
e^{i\xi^3\Delta t}
\begin{bmatrix}
1\\
i\xi\\
(i\xi)^2
\end{bmatrix}
+
e^{\Delta t\lambda_+}
\begin{bmatrix}
0\\
-\dfrac{2i\xi}{d_{+}}\\
\dfrac{\xi(\xi-s)}{d_{+}}
\end{bmatrix}
+
e^{\Delta t\lambda_-}
\begin{bmatrix}
0\\
-\dfrac{2i\xi}{d_{-}}\\
\dfrac{\xi(\xi+s)}{d_{-}}
\end{bmatrix}
\right)N^0
+\order(\tau).
\end{align*}
Therefore,
\begin{align}\label{Eq:LE_nonlinear_expansion}
\phi_0(\Delta tL(\tau))N(\qhat^0)
=
\begin{bmatrix}
e^{i\xi^3\Delta t}N^0\\
i\xi e^{i\xi^3\Delta t}N^0+\mu_1(\tau)\\
(i\xi)^2e^{i\xi^3\Delta t}N^0+\mu_2(\tau)
\end{bmatrix}
+\order(\tau),
\end{align}
where
\begin{align}
\mu_1(\tau)
=
\left(
-\frac{2i\xi}{d_{+}}e^{\Delta t\lambda_+}
-\frac{2i\xi}{d_{-}}e^{\Delta t\lambda_-}
\right)N^0,  \ 
\mu_2(\tau)
=
\left(
\frac{\xi(\xi-s)}{d_{+}}e^{\Delta t\lambda_+}
+\frac{\xi(\xi+s)}{d_{-}}e^{\Delta t\lambda_-}
\right)N^0 .
\end{align}
From the asymptotic expansion
\[
    \lambda_\pm
    =
    \frac{i\omega_\pm}{\tau}
    +\lambda_\pm^0
    +\order(\tau),
    \qquad
    \lambda_\pm^0
    =
    \frac{i\xi^2(\omega_\pm-\xi)}
         {3\omega_\pm^2-2\xi\omega_\pm-1},
    \qquad
    \omega_\pm
    =
    \frac{\xi\pm\sqrt{\xi^2+4}}{2},
\]
we have $e^{\Delta t\lambda_\pm} =e^{i\omega_\pm\Delta t/\tau}e^{\Delta t\lambda_\pm^0}
(1+\order(\tau))$. Thus $e^{\Delta t\lambda_\pm}=\order(1)$ and consequently $\mu_1(\tau),\mu_2(\tau)=\order(1)$ as $\tau\to0$.
Moreover, in general, $\mu_1(\tau)$ and $\mu_2(\tau)$ retain
nonvanishing rapidly oscillatory contributions and therefore do not
converge as $\tau\to0$.

Combining \eqref{Eq:LE_1st_term_lim} and
\eqref{Eq:LE_nonlinear_expansion}, we obtain
\begin{align}
    \uh^1
    &=
    e^{i\xi^3\Delta t}
    \left(\uh^0+\Delta t\,N^0\right)
    +o(1), \label{Eq:LE_u_one_step}\\
    \vh^1-i\xi\uh^1
    &=
    \Delta t\,\mu_1(\tau)+o(1), \label{Eq:LE_v_defect}\\
    \wh^1-(i\xi)^2\uh^1
    &=
    \Delta t\,\mu_2(\tau)+o(1). \label{Eq:LE_w_defect}
\end{align}
Hence, although the $u$-component has the correct limiting
Lawson--Euler update after one time step, the defects in the $v$- and $w$-components do not, in general, tend to zero as $\tau\to0$.
Thus the numerical solution does not approach the equilibrium manifold after one time step, and the Lawson--Euler method is not asymptotic preserving.
\end{proof}

 The above argument can be extended to higher-order Lawson methods: their nonlinear stage and update terms also involve only $\phi_0$-functions, which retain the rapidly oscillating modes as $\tau\to0$, and therefore, in general, the auxiliary variables fail to satisfy the limiting equilibrium relations. In the numerical section, we observe that in general, these schemes capture the correct reduced dynamics for $u$ component, even for long time behavior. In this sense, the method appears to be asymptotic preserving with respect to $u$. In contrast, for the auxiliary  variables $v$ and $w$, the scheme fails to satisfy the equilibrium relations. In order to understand this phenomenon, it requires a more detailed analysis, which goes beyond the scope of the present AP study for the exponential integrators. This issue will be investigated in a future work.   

\begin{thm}\label{Th:AP_Norsett_Euler}
    The Norsett–Euler method applied to the system \eqref{Eq:kdvh_Fourier_cont} with consistent initial data \eqref{Eq:well_prep} 
    is AP for all the components.
\end{thm}
\begin{proof}
Let $\qhat^n = [\uh^n, \vh^n, \wh^n]^T$ denote the approximate solution of $\qhat(\xi,t)$ at time $t=t_n$ in the Fourier space. By the well-prepared initial data assumption, we may, without loss of generality, assume that at time $t^n$, the numerical solution lies on the equilibrium manifold $\qhat^n = [\uh^n,\, i\xi\uh^n,\, (i\xi)^2\uh^n]^T$.
In order to establish the AP property, we show that, in the limit $\tau \to 0$, the Norsett–Euler method produces at the next time  $t^{n+1}$ a numerical solution $\uh^{n+1}$ consistent with the limit KdV equation, while  preserving the equilibrium manifold for the $v$ and $w$ component, i.e., we prove $\vh^{n+1} = i\xi\uh^{n+1}$,  $\wh^{n+1} =  (i\xi)^2\uh^{n+1}$.

The Norsett–Euler method computes the solution from a time step $t_n$ to $t_{n+1}$ as 
    \begin{align}
        \qhat^{n+1} = e^{\Delta t L(\tau)}\qhat^{n} + \Delta t \phi_{1}(\Delta t L(\tau))N(\qhat^{n}) \;.
    \end{align}
    The limiting discretization as $\tau \to 0$ is given by
    \begin{align}\label{Eq:NE_mthd_limit}
        \lim_{\tau \to 0} \qhat^{n+1} = \lim_{\tau \to 0} \left(e^{\Delta t L(\tau)}\qhat^{n} \right) + \Delta t \lim_{\tau \to 0} \left(\phi_{1}(\Delta t L(\tau))N(\qhat^{n})\right) \;.
    \end{align}
    As before, with the well-prepared initial data the first term becomes
    \begin{align}\label{Eq:NE_1st_term_lim}
        \lim_{\tau \to 0} \left(e^{\Delta t L(\tau)} \qhat^{n} \right) & =  
        e^{i \xi^3 \Delta t}\begin{bmatrix}
                                \uh^n \\
                                i \xi \uh^n \\
                                (i \xi)^2 \uh^n
                            \end{bmatrix}.
    \end{align}
    According to Lemma \ref{lem:phiproduct}, for the second term we obtain
    \begin{align}\label{Eq:NE_2nd_term_lim}
        \lim_{\tau \to 0} \left(\phi_{1}(\Delta t L(\tau))N(\qhat^{n})\right)  
                & = \begin{bmatrix}
                        \left(\frac{e^{i \xi^3 \Delta t}-1}{i \xi^3 \Delta t}\right) N(\xi,\uh^n) \\
                        i \xi \left(\frac{e^{i \xi^3 \Delta t}-1}{i \xi^3 \Delta t}\right) N(\xi,\uh^n)\\
                        (i \xi)^2 \left(\frac{e^{i \xi^3 \Delta t}-1}{i \xi^3 \Delta t}\right) N(\xi,\uh^n)
                    \end{bmatrix} \;.
    \end{align}  
    By inserting \eqref{Eq:NE_1st_term_lim} and \eqref{Eq:NE_2nd_term_lim} into \eqref{Eq:NE_mthd_limit}, we obtain the component-wise updates
    \begin{align}
        \uh^{n+1} &= e^{i \xi^3 \Delta t}\uh^n  
                    + \Delta t\, \left(\frac{e^{i \xi^3 \Delta t}-1}{i \xi^3 \Delta t}\right) N(\xi,\uh^n) \;, \label{Eq:NE_mthd_lim_comp_u} \\ 
        \vh^{n+1} &= i \xi e^{i \xi^3 \Delta t}\uh^n  
                    + i \xi \, \Delta t\, \left(\frac{e^{i \xi^3 \Delta t}-1}{i \xi^3 \Delta t}\right) N(\xi,\uh^n) \;, \label{Eq:NE_mthd_lim_comp_v} \\
        \wh^{n+1} &= (i \xi)^2 e^{i \xi^3 \Delta t}\uh^n  
                    + (i \xi)^2 \, \Delta t\, \left(\frac{e^{i \xi^3 \Delta t}-1}{i \xi^3 \Delta t}\right) N(\xi,\uh^n) \;, \label{Eq:NE_mthd_lim_comp_w}
    \end{align}
    Equation \eqref{Eq:NE_mthd_lim_comp_u} is exactly the update rule for the Nørsett–Euler method when applied to the KdV equation \eqref{kdv}, which proves the AP 
    property for the $u$-component. From equations \eqref{Eq:NE_mthd_lim_comp_v} and \eqref{Eq:NE_mthd_lim_comp_w}, we see that $\vh^{n+1} = i \xi \, \uh^{n+1}$ and 
    $\wh^{n+1} = (i \xi)^2 \, \uh^{n+1}$, which shows that the Nørsett–Euler method is also AP for both auxiliary components.
\end{proof}

We now establish the asymptotic-preserving property of ETDRK schemes of arbitrary order. We consider explicit ETDRK methods satisfying \eqref{1cond}, with $c_1=0$ and $c_i>0$ for $i=2,\dots,s$, whose coefficients $a_{ij}(z)$ and $b_i(z)$ are finite linear combinations of $\phi_k(c_\ell z)$ and $\phi_k(z)$, respectively, with $k\ge1$ and $c_\ell>0$. All ETDRK methods considered in this work satisfy these assumptions.

\begin{thm}\label{Th:AP_ETDRK_AA}
An $s$-stage ETDRK method of order $p$ satisfying the above assumptions, when applied to \eqref{Eq:kdvh_Fourier_cont} with well-prepared initial data \eqref{Eq:well_prep}, is asymptotic preserving (AP) for all components.
\end{thm}

\begin{proof}
Consider an $s$-stage ETDRK scheme of the form \eqref{ExInt2}
satisfying the above assumptions. When applied to the system \eqref{Eq:kdvh_Fourier_cont}, the scheme updates the solution vector $\qhat^n = [\uh^n, \vh^n, \wh^n]^T \approx \qhat(\xi, t_n)$ via
\begin{subequations}\label{Eq:ExpUpdateKdVH}
\begin{align}
    \Yh_i &= \phi_{0}(c_i Z) \qhat^n + \Delta t \sum_{j=1}^{i-1} a_{ij}(Z) N(\Yh_j), \quad i = 1, \dots, s, \label{Eq:ExpUpdateKdVH_a} \\ 
    \qhat^{n+1} &= \phi_{0}(Z) \qhat^n + \Delta t \sum_{i=1}^{s} b_{i}(Z) N(\Yh_i)\;, \label{Eq:ExpUpdateKdVH_b}
\end{align}
\end{subequations}
where $Z = \Delta t L(\tau)$ and $\Yh_i$ denotes the $i$-th stage vector approximation. To distinguish the
vector nonlinear term from its scalar first component, write
$N(\qhat)=e_1\mathcal{N}(\uh)$, where $e_1=[1,0,0]^T$ and $\mathcal{N}(\uh)$ is the nonlinear part of the kdV equation \eqref{Eq:kdv_Fourier_cont}.

For comparison, applying the same method directly to the scalar limiting KdV equation \eqref{Eq:kdv_Fourier_cont} in Fourier space yields
\begin{subequations}\label{Eq:ExpUpdateKdV}
\begin{align}
    \gh_i &= \phi_{0}(c_i \Delta t \, i\xi^3) \uh^n + \Delta t \sum_{j=1}^{i-1} a_{ij}(\Delta t \, i\xi^3) \mathcal{N}(\gh_j), \quad i = 1, \dots, s, \\
    \uh^{n+1} &= \phi_{0}(\Delta t \, i\xi^3) \uh^n + \Delta t \sum_{i=1}^{s} b_{i}(\Delta t \, i\xi^3) \mathcal{N}(\gh_i)\;,
\end{align}
\end{subequations}
where $\gh_i$ denotes the scalar $i$-th stage solution approximation.
We argue by induction over the time steps, keeping $\Delta t>0$
fixed. With a slight abuse of notation, let $\uh^n$ denote the numerical solution of the limiting
KdV equation and assume that
\[
\lim_{\tau\to0}\qhat^n
=
\begin{bmatrix}
\uh^n \\
i\xi\uh^n \\
(i\xi)^2\uh^n
\end{bmatrix}.
\]
This holds at $n=0$ by the well-prepared initial data assumption. We show that this also holds for the time level $t=t_{n+1}$.

From the definition \eqref{Eq:phi0_fun} and Lemma \ref{lem:covariants}, the first terms on the right-hand sides of \eqref{Eq:ExpUpdateKdVH_a} and \eqref{Eq:ExpUpdateKdVH_b} satisfy
\begin{align}\label{Eq:phi_0_tau_lim}
\lim_{\tau \to 0} \left(\phi_0(c_i Z)\hat{q}^{n}\right) = \phi_0(c_i \Delta t\, i\xi^3)
\begin{bmatrix} \uh^n \\ i\xi\uh^n \\ (i\xi)^2\uh^n \end{bmatrix},
\quad
\lim_{\tau \to 0} \left(\phi_0(Z)\hat{q}^{n}\right) = \phi_0(\Delta t \, i\xi^3)
\begin{bmatrix} \uh^n \\ i\xi\uh^n \\ (i\xi)^2\uh^n \end{bmatrix}.
\end{align}

We proceed by induction to demonstrate that for all stages $i = 1, 2, \dots, s$,
\begin{align}\label{Eq:KdVH_stage_tau_lim}
    \lim_{\tau \to 0} \Yh_i = \begin{bmatrix} \gh_i \\ i \xi \gh_i \\ (i \xi)^2 \gh_i \end{bmatrix}.
\end{align}

Since $c_1 = 0$ we have $\Yh_1 = \qhat^n$and $\gh_1=\uh^n$. Hence the time-step induction hypothesis establishes \eqref{Eq:KdVH_stage_tau_lim} for $i=1$.

Assume that \eqref{Eq:KdVH_stage_tau_lim} holds for all stage indices $j = 1, 2, \dots, i-1$. In particular, the first component of $\Yh_j$ converges to $\gh_j$;
hence, by continuity of the nonlinear pseudospectral operator,
$N(\Yh_j)=e_1\mathcal{N}((\Yh_j)_1)$ converges to 
$e_1\mathcal{N}(\gh_j)$ as $\tau\to0$. Because $a_{ij}(z)$ are finite linear combinations of
$\varphi_k(c_\ell z)$ with $k\ge1$ and $c_\ell>0$, applying the induction hypothesis alongside \eqref{Eq:phi_0_tau_lim} and Lemmas \ref{lem:phitau}--\ref{lem:phiproduct} with $\Delta t$
replaced by $c_\ell\Delta t$ in each coefficient term, yields
\begin{align}
    \lim_{\tau \to 0} \Yh_i &= \phi_0(c_i \Delta t\, i\xi^3)
    \begin{bmatrix} \uh^n \\ i\xi\uh^n \\ (i\xi)^2\uh^n \end{bmatrix} + \Delta t \sum_{j=1}^{i-1} a_{ij}(\Delta t \, i \xi^3) \begin{bmatrix} 1 \\ i\xi \\ (i\xi)^2 \end{bmatrix} \mathcal{N}(\gh_j) \nonumber \\
    &= \begin{bmatrix} 1 \\ i\xi \\ (i\xi)^2 \end{bmatrix} \left( \phi_0(c_i \Delta t\, i\xi^3) \uh^n + \Delta t \sum_{j=1}^{i-1} a_{ij}(\Delta t \, i \xi^3) \mathcal{N}(\gh_j) \right) = \begin{bmatrix} \gh_i \\ i \xi \, \gh_i \\ (i \xi)^2 \, \gh_i \end{bmatrix},
\end{align}
which establishes \eqref{Eq:KdVH_stage_tau_lim} for stage $i$.

Finally, taking the limit $\tau \to 0$ in the full update expression \eqref{Eq:ExpUpdateKdVH_b}, using the fact that $b_i(z)$ are linear combinations of $\varphi_k(z)$ ($k \ge 1$), gives
\begin{align}
    \lim_{\tau \to 0} \qhat^{n+1} &= \phi_0(\Delta t\, i\xi^3) \begin{bmatrix} \uh^n \\ i\xi \, \uh^n \\ (i\xi)^2 \, \uh^n \end{bmatrix} + \Delta t \sum_{i=1}^{s} b_{i}(\Delta t \, i \xi^3) \begin{bmatrix} 1 \\ i\xi \\ (i\xi)^2 \end{bmatrix} \mathcal{N}(\gh_i) \nonumber \\
    &= \begin{bmatrix} 1 \\ i\xi \\ (i\xi)^2 \end{bmatrix} \left( \phi_0(\Delta t\, i\xi^3) \uh^n + \Delta t \sum_{i=1}^{s} b_{i}(\Delta t \, i \xi^3) \mathcal{N}(\gh_i) \right) = \begin{bmatrix} \uh^{n+1} \\ i \xi \, \uh^{n+1} \\ (i \xi)^2 \, \uh^{n+1} \end{bmatrix}.
\end{align}

The first component confirms that the limit solution $\lim_{\tau \to 0} \uh^{n+1}$ identically matches the $s$-stage order $p$ ETDRK discretization of the limiting KdV equation. Furthermore, the second and third components establish $\lim_{\tau \to 0} \vh^{n+1} = i\xi \, \uh^{n+1}$ and $\lim_{\tau \to 0} \wh^{n+1} = (i\xi)^2 \uh^{n+1}$. Thus, the induction hypothesis holds at time level $n+1$.
This completes the proof of asymptotic preservation.
\end{proof}

\begin{remark}
The AP property does not in general imply asymptotic accuracy. An order-$p$ AP method may, in the limit $\tau\to0$, reduce to a consistent discretization of the limiting problem having order lower
than $p$. For the ETD methods considered in this work, however, our
numerical experiments indicate that the classical order is retained in
the stiff limit, so that these methods are also AA. The Lawson methods
are not AP for the full system and hence are not AA for the full system.
As discussed in Section~\ref{sec:tests}, stiff order appears instead to
play an important role in the accuracy for intermediate values of
$\tau$ and in uniform accuracy, although a precise characterization of
this relationship is beyond the scope of the present work.
\end{remark}

%==========================%
\section{Numerical results}
\label{sec:tests}
%=========================%
In this section, we verify the AP properties proved in the previous section and conduct additional investigations of the accuracy of exponential methods for the KdVH system with a solitary wave solution. We also present a comparison of computational efficiency of exponential methods relative to implicit–explicit (ImEx) methods for the KdVH system.

%\subsection{Numerical calculation of the $\phi$-functions}
% The implementation of exponential methods necessitates the accurate evaluation of $\phi$-functions, which is a non-trivial and well-studied challenge. The primary difficulty arises from cancellation errors when evaluating higher-order $\phi$-functions for small values of $z$. To mitigate this, a common strategy is to utilize a Taylor series expansion for small $z$ and the exact formulas for larger values. In our implementation, we adopt a switching threshold of $z = 10^{-8}$, an approach similar to that of Cox and Matthews \cite{cox2002exponential}. However, it is important to note that this method can occasionally exhibit inaccuracies near the switching point, depending on the application. While alternative techniques exist, such as those based on Pad\'{e} approximations or contour integration \cite{kassam2005fourth}, they each possess distinct advantages and drawbacks. For the purposes of this study, we first employ the Lagrange-Sylvester formula to compute the $\phi$-functions of the matrix exactly. For the subsequent scalar evaluations, we apply the switching strategy used by Cox and Matthews. In our numerical experiments, this combined approach consistently yielded the desired accuracy.

%The implementation of exponential integrators requires the evaluation of
%$\phi$-functions of matrices. 
In the present work, the matrix
$\phi$-functions are first reduced to scalar $\phi$-function evaluations
using the Lagrange--Sylvester formula, after which the resulting scalar
$\phi$-functions are computed using 
the Taylor series for 
$|z| < 10^{-8}$ and the standard closed-form expressions otherwise, as proposed in \cite{cox2002exponential}. All the code used to produce the numerical results presented in this paper can be found in \cite{biswas2026APExpKdVHRepro}.

%a switching strategy. The accurate
%evaluation of scalar $\phi$-functions is a well-known challenge,
%particularly for small values of the argument, where cancellation errors
%can occur.  Among the various approaches that have been proposed (see e.g. \cite{cox2002exponential,kassam2005fourth}), we have obtained sufficiently accurate results by using a Taylor series expansion when

\subsection{Tests of AP}
The setup for tests of the AP property is as follows.
Both the KdV equation and the KdVH system are solved on the domain $[x_L, x_R] \times (0, T] = [-40, 40] \times (0, 5]$ with periodic boundary conditions. We study the propagation of a soliton
%The initial condition for the KdV equation is obtained by evaluating the exact solution $\ukdv(x, t)$ at $t=0$, where
\begin{align} \label{soliton}
    \ukdv(x, t) = A \sech^2 \left( \frac{\sqrt{3A}(x - ct)}{6} \right),
\end{align}
with $A=3c$ and $c=1.2$. For KdV, the initial condition is obtained by setting $t=0$ in \eqref{soliton}.  For the KdVH system, we take $u(x, 0) = \ukdv(x, 0)$, $v(x, 0) = D \ukdv(x, 0)$, and $w(x, 0) = D^2 \ukdv(x, 0)$, where $D$ is the Fourier differential operator.  Thus the initial data for $v$ and $w$ are well-prepared to leading order (but not exactly).
We semidiscretize both problems using $2^9$ grid points via the pseudospectral Fourier collocation method, as described in Sec~\ref{sec:space}.
The resulting ODEs are integrated in Fourier space, and the final solution is obtained via the inverse FFT. To numerically verify the AP property of the schemes, we evaluate the errors in the $\ell_2$ norm by measuring the discrepancy between the fully discretized KdVH solution, denoted by $\bm{u}$, $\bm{v}$, and $\bm{w}$, and the numerical solution $\bm{\ukdv}$ of the KdV equation. Specifically, we compute the $\ell_2$ norm of the differences $\bm{u} - \bm{\ukdv}$, $\bm{v} - D \bm{\ukdv}$, and $\bm{w} - D^2 \bm{\ukdv}$. The experimental orders of convergence (EOCs) are computed using
consecutive values of $\tau$ differing by a factor of ten. For brevity,
Tables~2--8 report only every other value of $\tau$ and its corresponding EOC.
%, where $D$ represents the Fourier differential operator applied to the reference KdV solution. 
%------------------------------------------------------%
\subsubsection{IF Methods}
%------------------------------------------------------%
Here we test the three IF methods introduced in Section~\ref{sec:ERK}: Lawson--Euler (order~1, stiff order~1), Lawson2b (order~2, stiff order~1), and Lawson4 (order~4, stiff order~1). We use the setup described above with a time step of $\Delta t = 0.003$ for the Lawson--Euler method and $\Delta t = 0.015$ for the remaining methods. These step sizes are chosen to ensure stable computations up to the final time $t=5$.
The computed errors are given in Table~\ref{tab:ap_results_Lawson_Euler}, Table~\ref{tab:ap_results_Lawson2b}, and Table~\ref{tab:ap_results_Lawson4}. 
All three methods exhibit convergence as $\tau\to0$ for $u$, but not for $v$ or $w$.  This is consistent with our analysis, in the sense that Lawson methods are not AP.  It is worth noting that our analysis does not explain the fact that $u$ \emph{does} converge to the correct value.

% Lawson-Euler
\begin{table}[!tbp]
\centering
\captionsetup{width=0.9\textwidth} % Limits caption to 80% of page width
%\caption{AP test of Lawson-Euler. The $\ell_2$ norms of the errors are calculated relative to the numerical solution of the KdV equation.}
\caption{AP test of Lawson-Euler.}
\label{tab:ap_results_Lawson_Euler}
%\resizebox{\textwidth}{!}
\begin{tabular}{ccccccc}
  \hline
  \textbf{$\tau$} & \textbf{error u} & \textbf{EOC u} & \textbf{error v} & \textbf{EOC v} & \textbf{error w} & \textbf{EOC w} \\ \hline
  1.00e-02 & 2.32e-02 &  & 2.19e-02 &  & 2.90e-02 &  \\
  %1.00e-03 & 2.38e-03 & 0.99 & 1.87e-03 & 1.07 & 2.57e-03 & 1.05 \\
  1.00e-04 & 2.38e-04 & 1.00 & 1.11e-03 & 0.23 & 1.98e-03 & 0.11 \\
   %1.00e-05 & 2.42e-05 & 0.99 & 3.37e-03 & -0.48 & 3.89e-03 & -0.29 \\
  1.00e-06 & 2.47e-06 & 0.99 & 3.61e-03 & -0.03 & 9.10e-03 & -0.37 \\
   %1.00e-07 & 2.45e-07 & 1.00 & 1.74e-03 & 0.32 & 3.34e-03 & 0.44 \\
  1.00e-08 & 2.44e-08 & 1.00 & 2.38e-03 & -0.14 & 4.59e-03 & -0.14 \\
   %1.00e-09 & 2.47e-09 & 1.00 & 2.32e-03 & 0.01 & 4.46e-03 & 0.01 \\
  1.00e-10 & 2.44e-10 & 1.00 & 4.29e-03 & -0.27 & 1.20e-02 & -0.43 \\
  \hline
\end{tabular}
\end{table}
% Lawson2b
\begin{table}[!tbp]
\centering
\captionsetup{width=0.9\textwidth} % Limits caption to 80% of page width
\caption{AP test of Lawson2b.}

\label{tab:ap_results_Lawson2b}
%\resizebox{\textwidth}{!}
{\begin{tabular}{ccccccc}
  \hline
  \textbf{$\tau$} & \textbf{error u} & \textbf{EOC u} & \textbf{error v} & \textbf{EOC v} & \textbf{error w} & \textbf{EOC w} \\ \hline
  1.00e-02 & 2.29e-02 &  & 2.19e-02 &  & 2.91e-02 &  \\
   %1.00e-03 & 2.25e-03 & 1.01 & 3.20e-02 & -0.17 & 1.38e-02 & 0.32 \\
  1.00e-04 & 2.30e-04 & 0.99 & 9.23e-03 & 0.54 & 2.09e-02 & -0.18 \\
   %1.00e-05 & 2.31e-05 & 1.00 & 9.33e-03 & -0.00 & 1.37e-02 & 0.19 \\
  1.00e-06 & 2.29e-06 & 1.00 & 1.35e-02 & -0.16 & 3.24e-02 & -0.38 \\
   %1.00e-07 & 1.91e-07 & 1.08 & 5.61e-02 & -0.62 & 1.14e-01 & -0.54 \\
  1.00e-08 & 2.25e-08 & 0.93 & 1.35e-02 & 0.62 & 4.40e-02 & 0.41 \\
   %1.00e-09 & 2.27e-09 & 1.00 & 1.81e-02 & -0.13 & 4.39e-02 & 0.00 \\
  1.00e-10 & 2.51e-10 & 0.96 & 1.97e-02 & -0.04 & 4.98e-02 & -0.05 \\
  \hline
\end{tabular}}
\end{table}
% Lawson4
\begin{table}[!tbp]
\centering
\captionsetup{width=0.9\textwidth} % Limits caption to 80% of page width
\caption{AP test of Lawson4.}
\label{tab:ap_results_Lawson4}
%\resizebox{\textwidth}{!}
\begin{tabular}{ccccccc}
  \hline
  \textbf{$\tau$} & \textbf{error u} & \textbf{EOC u} & \textbf{error v} & \textbf{EOC v} & \textbf{error w} & \textbf{EOC w} \\ \hline
  1.00e-02 & 2.28e-02 &  & 2.19e-02 &  & 2.91e-02 &  \\
   %1.00e-03 & 2.34e-03 & 0.99 & 3.19e-02 & -0.16 & 1.37e-02 & 0.33 \\
  1.00e-04 & 2.31e-04 & 1.01 & 6.30e-03 & 0.70 & 1.33e-02 & 0.01 \\
   %1.00e-05 & 2.32e-05 & 1.00 & 5.58e-03 & 0.05 & 1.12e-02 & 0.08 \\
  1.00e-06 & 2.34e-06 & 1.00 & 5.20e-03 & 0.03 & 1.22e-02 & -0.04 \\
   %1.00e-07 & 2.00e-07 & 1.07 & 5.33e-02 & -1.01 & 1.09e-01 & -0.95 \\
  1.00e-08 & 2.27e-08 & 0.95 & 8.46e-03 & 0.80 & 2.48e-02 & 0.65 \\
   %1.00e-09 & 2.26e-09 & 1.00 & 1.75e-02 & -0.32 & 4.31e-02 & -0.24 \\
  1.00e-10 & 2.37e-10 & 0.98 & 1.71e-02 & 0.01 & 4.04e-02 & 0.03 \\
  \hline
\end{tabular}
\end{table}
%------------------------------------------------------%
\subsubsection{ETD Methods}
%------------------------------------------------------%
Next, we test ETD methods. From this class, we present the results for four specific schemes: N{\o}rsett--Euler (order 1, stiff order 1), ETD2RK (order 2, stiff order 2), ETD3RK (order 3, stiff order 2), and Hochbruck-Ostermann (order 4, stiff order 4). These methods are applied to both semidiscrete systems \eqref{Eq:kdv_Fourier} and \eqref{semidiscretization} using $m = 2^9$ grid points, and integrated up to the final time $T = 5$. We use a time step of $\Delta t = 0.015$ for all the methods. This step size is chosen to ensure stable computation of the numerical solutions. The asymptotic errors for these methods are given in Table~\ref{tab:ap_results_Norsett_Euler}, Table~\ref{tab:ap_results_ETD2RK},  Table~\ref{tab:ap_results_ETD3RK}, and
Table~\ref{tab:ap_results_HochbruckOstermann}.
As predicted by our analysis, all of these methods exhibit the AP property for all components of the system.
% N{\o}rsett--Euler
\begin{table}[!tbp]
\centering
\captionsetup{width=0.9\textwidth} % Limits caption to 80% of page width
\caption{AP test of N{\o}rsett--Euler.}
\label{tab:ap_results_Norsett_Euler}
%\resizebox{\textwidth}{!}
\begin{tabular}{ccccccc}
  \hline
  \textbf{$\tau$} & \textbf{error u} & \textbf{EOC u} & \textbf{error v} & \textbf{EOC v} & \textbf{error w} & \textbf{EOC w} \\ \hline
  1.00e-02 & 2.28e-02 &  & 2.19e-02 &  & 2.91e-02 &  \\
   %1.00e-03 & 2.35e-03 & 0.99 & 2.18e-03 & 1.00 & 2.97e-03 & 0.99 \\
  1.00e-04 & 2.35e-04 & 1.00 & 2.20e-04 & 1.00 & 2.95e-04 & 1.00 \\
   %1.00e-05 & 2.35e-05 & 1.00 & 2.17e-05 & 1.01 & 2.92e-05 & 1.00 \\
  1.00e-06 & 2.35e-06 & 1.00 & 2.20e-06 & 0.99 & 3.03e-06 & 0.98 \\
   %1.00e-07 & 2.35e-07 & 1.00 & 2.14e-07 & 1.01 & 2.90e-07 & 1.02 \\
  1.00e-08 & 2.35e-08 & 1.00 & 2.12e-08 & 1.00 & 2.81e-08 & 1.01 \\
   %1.00e-09 & 2.36e-09 & 1.00 & 2.03e-09 & 1.02 & 2.72e-09 & 1.01 \\
  1.00e-10 & 2.32e-10 & 1.01 & 2.15e-10 & 0.98 & 2.96e-10 & 0.96 \\
  \hline
\end{tabular}
\end{table}
% ETD2RK
\begin{table}[!tbp]
\captionsetup{width=0.9\textwidth} % Limits caption to 80% of page width
\centering
\caption{AP test of ETD2RK.}
\label{tab:ap_results_ETD2RK}
%\resizebox{\textwidth}{!}
\begin{tabular}{ccccccc}
  \hline
  \textbf{$\tau$} & \textbf{error u} & \textbf{EOC u} & \textbf{error v} & \textbf{EOC v} & \textbf{error w} & \textbf{EOC w} \\ \hline
  1.00e-02 & 2.28e-02 &  & 2.19e-02 &  & 2.91e-02 &  \\
   %1.00e-03 & 2.35e-03 & 0.99 & 2.18e-03 & 1.00 & 2.97e-03 & 0.99 \\
  1.00e-04 & 2.35e-04 & 1.00 & 2.20e-04 & 1.00 & 2.95e-04 & 1.00 \\
   %1.00e-05 & 2.35e-05 & 1.00 & 2.17e-05 & 1.01 & 2.93e-05 & 1.00 \\
  1.00e-06 & 2.35e-06 & 1.00 & 2.20e-06 & 0.99 & 3.04e-06 & 0.98 \\
   %1.00e-07 & 2.35e-07 & 1.00 & 2.14e-07 & 1.01 & 2.91e-07 & 1.02 \\
  1.00e-08 & 2.35e-08 & 1.00 & 2.12e-08 & 1.00 & 2.83e-08 & 1.01 \\
   %1.00e-09 & 2.35e-09 & 1.00 & 2.03e-09 & 1.02 & 2.74e-09 & 1.01 \\
  1.00e-10 & 2.35e-10 & 1.00 & 2.18e-10 & 0.97 & 2.97e-10 & 0.96 \\
  \hline
\end{tabular}
\end{table}
% ETD3RK
\begin{table}[!tbp]
\centering
\captionsetup{width=0.9\textwidth} % Limits caption to 80% of page width
\caption{AP test of ETD3RK.}
\label{tab:ap_results_ETD3RK}
%\resizebox{\textwidth}{!}
\begin{tabular}{ccccccc}
  \hline
  \textbf{$\tau$} & \textbf{error u} & \textbf{EOC u} & \textbf{error v} & \textbf{EOC v} & \textbf{error w} & \textbf{EOC w} \\ \hline
  1.00e-02 & 2.28e-02 &  & 2.19e-02 &  & 2.91e-02 &  \\
   %1.00e-03 & 2.34e-03 & 0.99 & 2.18e-03 & 1.00 & 2.97e-03 & 0.99 \\
  1.00e-04 & 2.35e-04 & 1.00 & 2.20e-04 & 1.00 & 2.95e-04 & 1.00 \\
   %1.00e-05 & 2.35e-05 & 1.00 & 2.17e-05 & 1.01 & 2.92e-05 & 1.00 \\
  1.00e-06 & 2.35e-06 & 1.00 & 2.20e-06 & 0.99 & 3.03e-06 & 0.98 \\
   %1.00e-07 & 2.35e-07 & 1.00 & 2.14e-07 & 1.01 & 2.90e-07 & 1.02 \\
  1.00e-08 & 2.35e-08 & 1.00 & 2.12e-08 & 1.00 & 2.81e-08 & 1.01 \\
   %1.00e-09 & 2.36e-09 & 1.00 & 2.03e-09 & 1.02 & 2.72e-09 & 1.01 \\
  1.00e-10 & 2.42e-10 & 0.99 & 2.20e-10 & 0.96 & 3.00e-10 & 0.96 \\
  \hline
\end{tabular}
\end{table}
% %ETD4RK
% \begin{table}[!tbp]
% \centering
% \caption{AP test of ETD4RK. The $\ell_2$ norms of the errors are calculated relative to the numerical solution of the KdV equation.}
% \label{tab:ap_results_ETD4RK}
% %\resizebox{\textwidth}{!}
% {\begin{tabular}{ccccccc}
%   \hline
%   \textbf{$\tau$} & \textbf{error u} & \textbf{EOC u} & \textbf{error v} & \textbf{EOC v} & \textbf{error w} & \textbf{EOC w} \\ \hline
%   1.00e-02 & 2.28e-02 &  & 2.19e-02 &  & 2.91e-02 &  \\
%   %1.00e-03 & 2.34e-03 & 0.99 & 2.18e-03 & 1.00 & 2.97e-03 & 0.99 \\
%   1.00e-04 & 2.35e-04 & 1.00 & 2.20e-04 & 1.00 & 2.95e-04 & 1.00 \\
%   %1.00e-05 & 2.35e-05 & 1.00 & 2.17e-05 & 1.01 & 2.93e-05 & 1.00 \\
%   1.00e-06 & 2.35e-06 & 1.00 & 2.20e-06 & 0.99 & 3.04e-06 & 0.98 \\
%   %1.00e-07 & 2.35e-07 & 1.00 & 2.14e-07 & 1.01 & 3.09e-07 & 0.99 \\
%   1.00e-08 & 2.35e-08 & 1.00 & 2.13e-08 & 1.00 & 3.02e-08 & 1.01 \\
%   %1.00e-09 & 2.35e-09 & 1.00 & 2.03e-09 & 1.02 & 2.93e-09 & 1.01 \\
%   1.00e-10 & 2.34e-10 & 1.00 & 2.18e-10 & 0.97 & 3.15e-10 & 0.97 \\
%   \hline
% \end{tabular}}
% \end{table}
% Hochbruck-Ostermann
\begin{table}[!tbp]
\centering
\caption{AP test of Hochbruck-Ostermann.}
\label{tab:ap_results_HochbruckOstermann}
%\resizebox{\textwidth}{!}
{\begin{tabular}{ccccccc}
  \hline
  \textbf{$\tau$} & \textbf{error u} & \textbf{EOC u} & \textbf{error v} & \textbf{EOC v} & \textbf{error w} & \textbf{EOC w} \\ \hline
  1.00e-02 & 2.28e-02 &  & 2.19e-02 &  & 2.91e-02 &  \\
  %1.00e-03 & 2.34e-03 & 0.99 & 2.18e-03 & 1.00 & 2.97e-03 & 0.99 \\
  1.00e-04 & 2.35e-04 & 1.00 & 2.20e-04 & 1.00 & 2.95e-04 & 1.00 \\
  %1.00e-05 & 2.35e-05 & 1.00 & 2.17e-05 & 1.01 & 2.92e-05 & 1.00 \\
  1.00e-06 & 2.35e-06 & 1.00 & 2.20e-06 & 0.99 & 3.03e-06 & 0.98 \\
  %1.00e-07 & 2.35e-07 & 1.00 & 2.14e-07 & 1.01 & 2.90e-07 & 1.02 \\
  1.00e-08 & 2.35e-08 & 1.00 & 2.12e-08 & 1.00 & 2.82e-08 & 1.01 \\
  %1.00e-09 & 2.35e-09 & 1.00 & 2.03e-09 & 1.02 & 2.72e-09 & 1.01 \\
  1.00e-10 & 2.35e-10 & 1.00 & 2.17e-10 & 0.97 & 2.96e-10 & 0.96 \\
  \hline
\end{tabular}}
\end{table}
% % ETD5RKF
% \begin{table}[!tbp]
% \centering
% \caption{AP test of the ETD5RKF method for the KdVH system. The $\ell_2$ norms of the errors are calculated relative to the numerical solution of the KdV equation.}
% \label{tab:ap_results_ETD5RKF}
% %\resizebox{\textwidth}{!}
% {\begin{tabular}{ccccccc}
%   \hline
%   \textbf{$\tau$} & \textbf{error u} & \textbf{EOC u} & \textbf{error v} & \textbf{EOC v} & \textbf{error w} & \textbf{EOC w} \\ \hline
%   1.00e-02 & 2.28e-02 &  & 2.19e-02 &  & 2.91e-02 &  \\
%   %1.00e-03 & 2.34e-03 & 0.99 & 2.18e-03 & 1.00 & 2.97e-03 & 0.99 \\
%   1.00e-04 & 2.35e-04 & 1.00 & 2.20e-04 & 1.00 & 2.95e-04 & 1.00 \\
%   %1.00e-05 & 2.35e-05 & 1.00 & 2.17e-05 & 1.01 & 2.92e-05 & 1.00 \\
%   1.00e-06 & 2.35e-06 & 1.00 & 2.20e-06 & 0.99 & 3.03e-06 & 0.98 \\
%   %1.00e-07 & 2.35e-07 & 1.00 & 2.14e-07 & 1.01 & 2.90e-07 & 1.02 \\
%   1.00e-08 & 2.35e-08 & 1.00 & 2.12e-08 & 1.00 & 2.81e-08 & 1.01 \\
%   %1.00e-09 & 2.36e-09 & 1.00 & 2.03e-09 & 1.02 & 2.72e-09 & 1.01 \\
%   1.00e-10 & 2.34e-10 & 1.00 & 2.17e-10 & 0.97 & 2.95e-10 & 0.96 \\
%   \hline
% \end{tabular}}
% \end{table}

%------------------------------------------------------%
\subsection{Tests of asymptotic and uniform accuracy}\label{sec:AAtests}
%------------------------------------------------------%

The AP property guarantees that the solution converges to that of the limit system as $\tau \to 0$, but it does not indicate the numerical convergence rate that will be observed as $\Delta t \to 0$ for different values of $\tau$. For stiff systems, it is known that the observed convergence rate may deteriorate for moderate or highly stiff parameter values \cite{boscarino2009class}. In this section, we investigate how the temporal convergence rate depends on the relaxation parameter $\tau$.

As mentioned above, a scheme is said to be asymptotically accurate if it retains its formal order of accuracy in the stiff limit $\tau \to 0$. Furthermore, if the order of accuracy is preserved uniformly for all values of $\tau$, then the scheme is said to be uniformly accurate.
To assess the asymptotic and uniform accuracy of the exponential methods considered in this work, we perform temporal convergence studies for several values of the relaxation parameter $\tau$, ranging from $10^{-2}$ to $10^{-8}$. For each fixed value of $\tau$, the KdVH system is discretized in space using $m = 2^9$ grid points so that spatial discretization errors are negligible. We then solve the resulting semidiscrete system using the Lawson2b, Lawson4, ETD2RK, ETD4RK, and Hochbruck--Ostermann methods. For each method, a sequence of time step sizes is considered, and the resulting errors are used to determine the observed temporal orders of convergence for all solution components.

In Figure \ref{fig:AA_property_Lawson2b} we show results for the IF method Lawson2b.  We see consistent second-order convergence of $u$ for all values of $\tau$.  In contrast, second-order convergence of $v$ and $w$ is observed only in the non-stiff regime, when $\Delta t \ll \tau$.
For Lawson4, in Figure \ref{fig:AA_property_Lawson4}, results for $v$ and $w$ are similar.  Convergence of $u$ is more erratic than for Lawson2b; a consistent 4th-order rate is seen only for the largest and smallest values of $\tau$.

For ETD methods, shown in Figures \ref{fig:AA_property_ETD2RK}-\ref{fig:AA_property_HochbruckOstermann}, we see consistent convergence of all components at the expected rate.

\begin{figure}[!tbp]
    \centering
        \centering
        \includegraphics[width=1.0\textwidth]{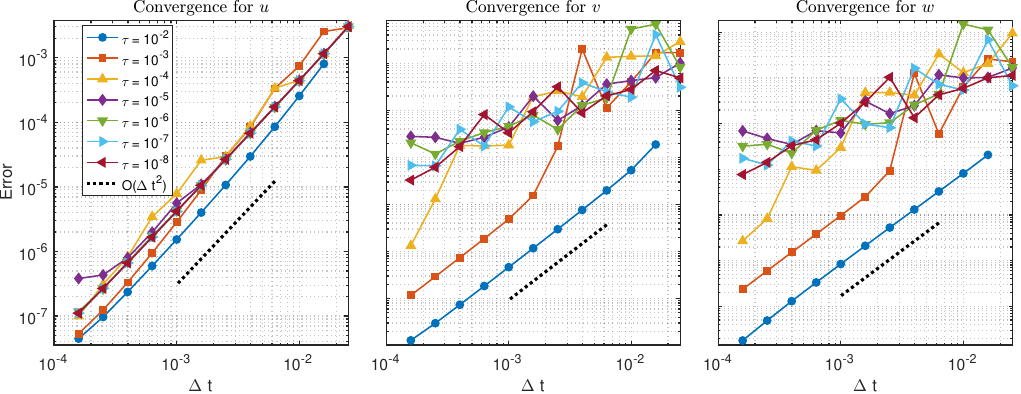}
    \caption{Numerical results illustrating the asymptotic and uniform accuracy properties of the second-order Lawson2b exponential integrator (stiff order one). The method exhibits uniform second-order convergence for the $u$-component, while the auxiliary components attain second-order convergence only in the nonstiff regime.}
    \label{fig:AA_property_Lawson2b}
\end{figure}
\begin{figure}[!tbp]
    \centering
        \centering
        \includegraphics[width=1.0\textwidth]{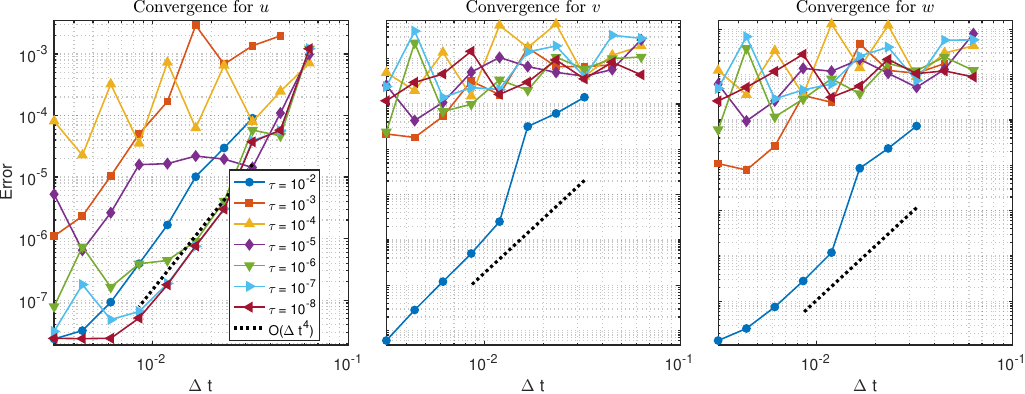}
    \caption{Numerical results illustrating the asymptotic and uniform accuracy properties of the fourth-order Lawson4 exponential integrator (stiff order one). The method attains fourth-order convergence for the $u$-component in both the nonstiff regime and the stiff limit, while order reduction is observed for intermediate values of $\tau$. The auxiliary components attain the formal order of accuracy only in the nonstiff regime.}
    \label{fig:AA_property_Lawson4}
\end{figure}
\begin{figure}[!tbp]
    \centering
        \centering
        \includegraphics[width=1.0\textwidth]{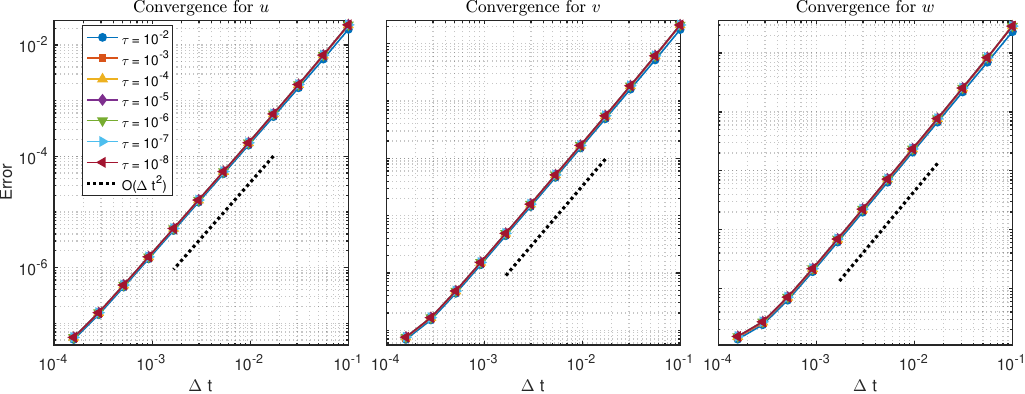}
    \caption{Numerical results illustrating the asymptotic and uniform accuracy properties of the second-order ETD2RK method (stiff order two). The method exhibits uniform second-order convergence for all components across all regimes.}
    \label{fig:AA_property_ETD2RK}
\end{figure}
\begin{figure}[htbp]
    \centering
    \includegraphics[width=1.0\textwidth]
    {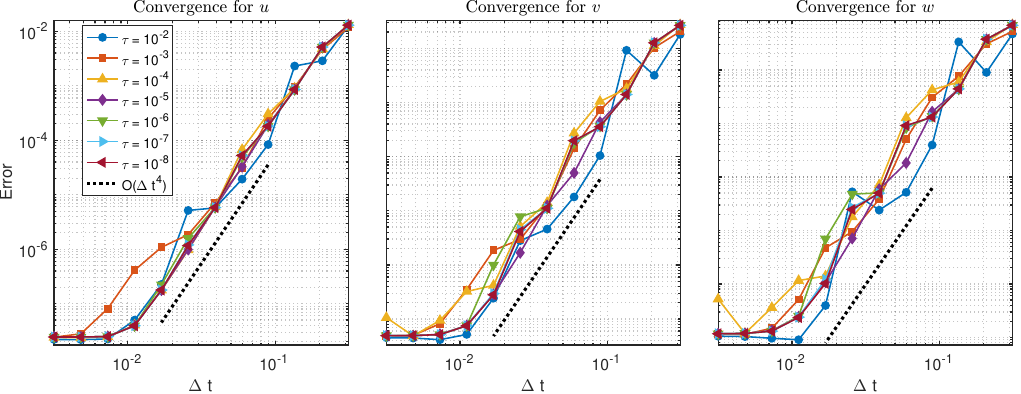}
    \caption{Numerical results illustrating the uniform accuracy properties of the fourth-order ETD4RK method (stiff order two). The method exhibits uniform fourth-order convergence for all components across all regimes.}
    \label{fig:AA_property_ETD4RK}
\end{figure}
\begin{figure}[!tbp]
    \centering
        \centering
        \includegraphics[width=1.0\textwidth]{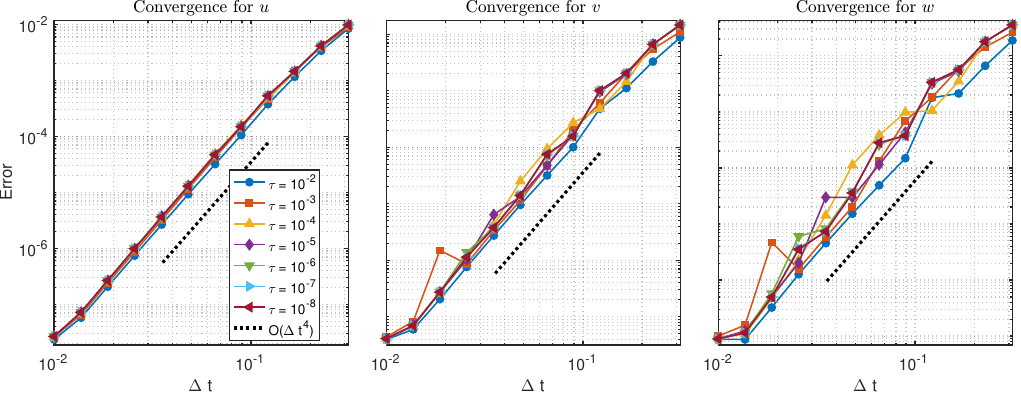}
    \caption{Numerical results illustrating the uniform accuracy properties of the fourth-order Hochbruck--Ostermann method (stiff order four). The method exhibits uniform fourth-order convergence for all components across all regimes.}
    \label{fig:AA_property_HochbruckOstermann}
\end{figure}
Note that the stagnation of the error curves at lower thresholds is due to a combination of three factors: the spatial discretization and floating-point roundoff error which is amplified to $\mathcal{O}(\varepsilon_{\text{mach}}/\tau)$ due to division by $\tau$, the algebraic tolerance of the Petviashvili method used to generate the reference solution, and the residual error accumulation from the time-stepping methods. 

The numerical results suggest that, similar to hyperbolic-to-hyperbolic relaxation systems, the AP property is closely related to asymptotic accuracy for the hyperbolic-to-dispersive relaxation problem considered in this work. In particular, methods possessing the AP property for a given component appear to recover their formal order of accuracy in the stiff limit. However, asymptotic accuracy alone does not guarantee uniform accuracy across all values of the relaxation parameter. Our experiments indicate that stiff order conditions play an important role in achieving uniform accuracy. In particular, the ETD methods with stiff order two or higher exhibit the expected convergence rates uniformly over the range of $\tau$ values considered, whereas methods with stiff order one may suffer from order reduction for intermediate values of $\tau$. These observations suggest a strong connection between stiff order conditions and uniform accuracy for exponential integrators applied to hyperbolic-to-dispersive relaxation systems. A rigorous analytical investigation of this relationship is beyond the scope of the present work and is left for future study.
%---------------------------------------------------------------------------%
\subsection{Cost comparison between exponential integrators and ImEx methods}
%---------------------------------------------------------------------------%
In this section we present a brief comparison of exponential and ImEx methods.  Such comparison is fraught with complications: the choice of methods for computing the matrix exponential and solving linear systems, the language of implementation, and numerous numerical parameters such as the temporal and spatial mesh sizes, the initial data, etc.  All of these can have a substantial influence on the results, so one should not draw overly general conclusions.  Nevertheless, a point of comparison still seems useful in establishing a rough idea of the competitiveness of the different methods.

% \david{Abhijit, explain the problem setup and what is measured, the hardware, etc.}

% In Figure \ref{} we provide work-efficiency diagrams for a range of methods of orders two to four.  Broadly speaking, exponential methods are competitive with ImEx methods, and can be more efficient at lower orders and/or with loose error tolerance.  At high order and tight tolerances, we note the remarkable efficiency of the relatively recent method ARK4(3)7L[2]SA.

We consider the KdVH system with $\tau = 10^{-5}$ on the spatial domain $[x_L, x_R] = [-40, 40]$. The system is initialized with a well-prepared traveling wave profile and semi-discretized using a Fourier collocation method on a uniform grid of $2^{10}$ spatial points to ensure that spatial truncation errors remain negligible. We compare ImEx and exponential integrators of the same order of accuracy by integrating up to a final time $t = 1$ using a sequence of uniform temporal step sizes $\Delta t$. For each configuration, we record both the global error at the final time and the computational cost. To eliminate random operating system background noise and ensure reproducible timing metrics, each simulation is executed five consecutive times, with the final recorded wall-clock runtime taken as the average over these runs. All computations are performed using a sequential implementation in MATLAB (version R2026a) on a workstation equipped with an Apple M4 Pro processor running under macOS Sequoia. To guarantee a fair comparison and eliminate multi-threading distortions during performance profiling, MATLAB execution is explicitly restricted to a single computational thread.

In Figure~\ref{fig:Cost_Comp_Exp_vs_ImEx}, we provide work-efficiency diagrams for a selection of numerical methods of orders two through four by plotting the $L^2$ error against the computational runtime on a log-log scale. Broadly speaking, exponential methods remain highly competitive with ImEx schemes, often exhibiting superior efficiency at lower orders or under looser error tolerances. Conversely, at high order and tight tolerances, we note the remarkable efficiency of the relatively recent ImEx method ARK4(3)7L[2]SA.

%\abhi{I have updated the cost–comparison plots. The new results are obtained by fixing the spatial resolution at a fine mesh of $2^9$ grid points and varying the time step size. These results are consistent across different ranges of time steps, unlike the previous results, which were sensitive to the CFL number. We still observe that the error stagnates around $10^{-9}$. 
%I checked the spatial discretization error (at $t = 0$) in the linear implicit part (the error in the explicit part is very small) and found that, for the $v$ and $w$ components, it is around $10^{-10}$ when $\tau = 10^{-5}$ and around $10^{-5}$ when $\tau = 10^{-10}$. This explains the error stagnation observed for $\tau = 10^{-5}$ (these errors are expected, since the linear implicit part involves division by a small $\tau$). However, I am not fully confident about the reliability of the results for $\tau = 10^{-10}$. 
%I also repeated the same experiment with a larger number of spatial grid points, but it appears that the spatial discretization error is already minimized for $2^9$ grid points.}
\begin{figure}
     \centering
     % First Row
     \begin{subfigure}{0.31\textwidth}
         \centering
         \includegraphics[width=\textwidth]{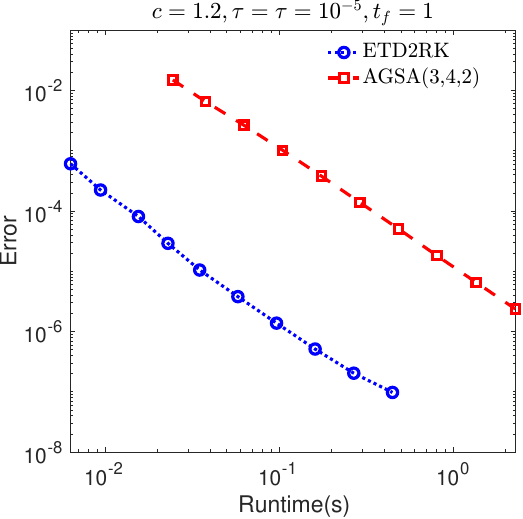}
     \end{subfigure}%
     \hspace{0.01\textwidth} % Reduced from 0.05 to 0.02
     \begin{subfigure}{0.31\textwidth}
         \centering
         \includegraphics[width=\textwidth]{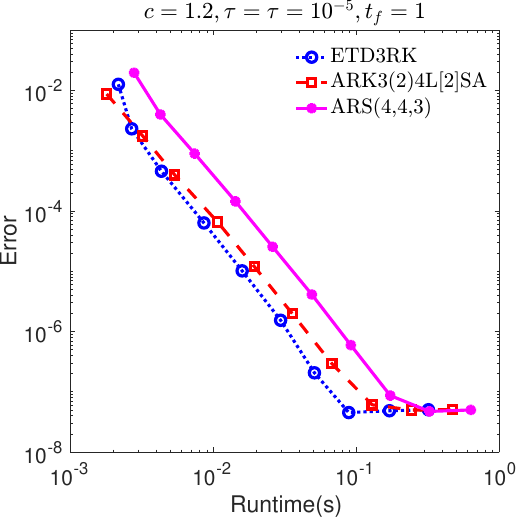}
     \end{subfigure}%
     \hspace{0.01\textwidth} % Reduced from 0.05 to 0.02
     \begin{subfigure}{0.31\textwidth}
         \centering
         \includegraphics[width=\textwidth]{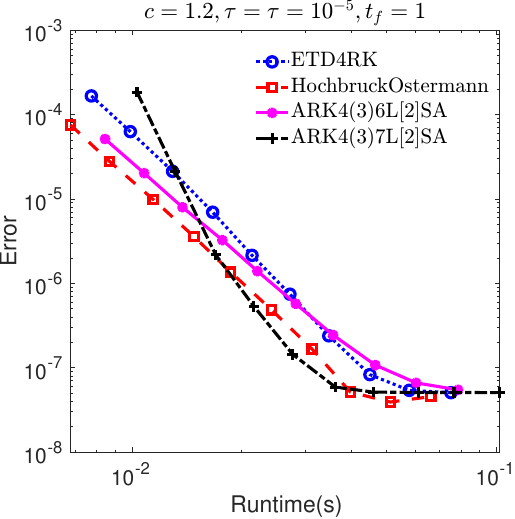}
     \end{subfigure}
     \caption{Cost measurement comparison for the $u$ component of the KdVH with $\tau = 10^{-5}$, between exponential methods and ImEx methods: (a) left panel for 2nd-order methods, (b) middle panel for 3rd-order methods, and (c) right panel for 4th-order methods. The errors in all these plots are measured with respect to the reference solutions computed using a Petviashvili-type method.}
     \label{fig:Cost_Comp_Exp_vs_ImEx}
 \end{figure}

%-----------------------------%
\section{Conclusions}
%-----------------------------%
With the goal of understanding the behavior of exponential integrators for nondissipative hyperbolic relaxation systems, we have considered the hyperbolic Korteweg--de Vries (KdVH) system as a representative example. Our analysis shows a clear distinction between Lawson and ETD methods in the relaxation limit. Lawson methods are, in general, not asymptotic preserving (AP) for the full KdVH system, since the auxiliary variables do not approach the equilibrium relations in the stiff limit. In contrast, the ETD methods considered here are AP for all components.

The essential reason for this difference lies in the distinct behavior of $\phi_0$ and the higher-order $\phi_k$ functions.  As shown in Lemmas \ref{lem:covariants} and \ref{lem:phitau}, $\phi_0(\Delta t L(\tau))$ retains contributions from the fast oscillatory modes of the hyperbolized system as $\tau\to0$, whereas for $k>0$ these contributions vanish in the corresponding limit of $\phi_k(\Delta t L(\tau))$.  This explains, in particular, why N{\o}rsett--Euler, which involves $\phi_1$, is AP, while Lawson--Euler, whose evolution operator is $\phi_0$, is not.  More generally, this mechanism suggests that the asymptotic behavior of other exponential integrators can often be anticipated from whether, and how, $\phi_0$ and the higher-order $\phi_k$ functions enter the method.

An interesting numerical observation is that, although Lawson methods are not AP for the full system, the $u$-component still converges to the limiting KdV solution in the stiff limit. This behavior is not explained by the present analysis and merits further investigation. The experiments also highlight the importance of well-prepared initial data in the nondissipative setting, where undamped fast modes can generate persistent oscillations. Higher-order well-prepared data substantially reduce these effects.  

Beyond the AP property, our numerical results suggest a connection between stiff order and uniform accuracy across different relaxation regimes. ETD methods with higher stiff order recover the expected convergence rates more uniformly over the range of $\tau$ considered, while methods with lower stiff order exhibit order reduction that depends on the relaxation parameter. A rigorous characterization of this relationship remains open.

For the KdVH system, we evaluate the required matrix exponentials and higher-order $\phi$-functions efficiently using the Lagrange--Sylvester formula. Our cost comparisons show that exponential methods are competitive with ImEx Runge--Kutta methods, although their relative efficiency depends on the implementation and several other factors. Although our analysis is specific to the KdVH system with Fourier pseudospectral discretization, the results may provide useful insight into exponential integration for more general nondissipative hyperbolic relaxation systems.

\appendix

\section{Well prepared initial data}
\label{sec:well_prepared_order_1}
%------------------%
We consider an asymptotic expansion of the solution of the KdVH system
\eqref{Eq:kdvH} in $\tau$,
\begin{equation}
U =
\begin{pmatrix}
u\\
v\\
w
\end{pmatrix}
=
U^0+\tau U^1+\tau^2U^2+\cdots .
\end{equation}
Retaining terms up to first order in $\tau$ and substituting into
\eqref{Eq:kdvH}, we obtain
\begin{align*}
u_t^0+\tau u_t^1
+(u^0+\tau u^1)(u_x^0+\tau u_x^1)
+w_x^0+\tau w_x^1 = O(\tau^2),\\
\tau(v_t^0+\tau v_t^1)
=v_x^0+\tau v_x^1-w^0-\tau w^1+O(\tau^2),\\
\tau(w_t^0+\tau w_t^1)
=v^0+\tau v^1-u_x^0-\tau u_x^1+O(\tau^2).
\end{align*}

At order $O(1)$, we have
\begin{equation}
u_t^0+u^0u_x^0+w_x^0=0,
\qquad
v_x^0=w^0,
\qquad
v^0=u_x^0.
\end{equation}
At order $O(\tau)$, we obtain
\begin{equation}
u_t^1+u^0u_x^1+u^1u_x^0+w_x^1=0,
\qquad
v_t^0=v_x^1-w^1,
\qquad
w_t^0=v^1-u_x^1.
\end{equation}

If we require the $u$-component of the KdVH solution to have the same
initial profile as the KdV equation for every $\tau$, then we take
$u^j(x,0)=0$ for all $j\geq1$. The first-order initial data for $v$ and $w$ are then determined by $v^1=w_t^0=v_{tx}^0$ and $w^1=v_x^1-v_t^0$, where
$v_t^0=u_{xt}^0=-(u^0u_x^0+w_x^0)_x$.

\vspace{1.5em}
\bibliographystyle{plain}
\bibliography{refs}
%===========================================================================

\vspace{1.5em}
\end{document}